\documentclass[a4paper]{cas-sc}

\usepackage[authoryear,longnamesfirst]{natbib}

\usepackage{caption}
\usepackage{subcaption}
\usepackage{setspace}
\usepackage{graphicx}
\usepackage{float}

\usepackage{algorithm}
\usepackage{amssymb}
\usepackage{amsfonts}
\usepackage{amsmath}
\usepackage{amsthm}

\usepackage{mathrsfs}
\usepackage{mathtools}
\usepackage{listings}
\usepackage{multirow}
\usepackage{comment}

\usepackage{tikz}
\usepackage{tikz-cd}
\usepackage[all,pdf]{xy}

\usepackage{algorithmic}
\usepackage{array}
\usepackage{lipsum}
\usepackage{hyperref}
\usepackage{url}

\usepackage{chngcntr}
\counterwithin{algorithm}{section}
\usepackage{booktabs}
\usepackage{color}

\renewcommand{\vec}[1]{\boldsymbol{#1}}

\newtheorem{thm}{Theorem}[section]
\newtheorem{assump}[thm]{Assumption}

\newtheorem{lem}[thm]{Lemma}
\newtheorem{rmk}[thm]{Remark}
\newtheorem{prop}[thm]{Proposition}

\numberwithin{equation}{section}

\usepackage{pgffor}
\foreach \x in {A,...,Z}{
\expandafter\xdef\csname \x cal\endcsname{{\noexpand\mathcal{\x}}}
}

\foreach \x in {A,...,Z}{
\expandafter\xdef\csname \x bf\endcsname{{\noexpand\mathbf{\x}}}
}

\foreach \x in {A,...,Z}{
\expandafter\xdef\csname \x bb\endcsname{{\noexpand\mathbb{\x}}}
}

\newcommand{\p}{\partial}

\newcommand{\diag}{\text{\rm diag}}

\newcommand{\spann}{\text{\rm span}}

\newcommand{\image}{\text{Im}}

\allowdisplaybreaks % For long formula

\newcommand{\normsubscript}{2}
\DeclarePairedDelimiter{\norm}{\|}{\|}
\newcommand{\normtwo}[1]{\norm*{#1}_{\normsubscript}}

\newcommand{\wnk}{\vec{\mathrm{k}}}
\newcommand{\wnkk}{\mathrm{k}}

\newcommand{\conjT}{\dagger}

\newcommand{\imagunit}{\mathrm{i}}

\newcommand{\mattrivial}{\Mcal_{\text{Trivial}}}
\newcommand{\matcrossdof}{\Mcal_{\text{CrossDoF}}}

\allowdisplaybreaks

\def\tsc#1{\csdef{#1}{\textsc{\lowercase{#1}}\xspace}}
\tsc{WGM}
\tsc{QE}
\begin{document}
\let\WriteBookmarks\relax
\def\floatpagepagefraction{1}
\def\textpagefraction{.001}

\shorttitle{}
\shortauthors{C. Jin and H. Xie}

% Main title
\title [mode = title]{A Robust GPU-Accelerated Kernel Compensation Solver with Novel Discretization for Photonic Crystals in Anisotropic Media}

% footnote
\tnotemark[1]

% fund
\tnotetext[1]{This work was supported by Strategic Priority Research Program of the Chinese Academy of Sciences (XDB0640000, XDB0640300, XDB0620203, XDA0480504), National Key Research and Development Program of China (2023YFB3309104), Natural Science Foundations of China (No. 1233000214), National Key Laboratory of Computational Physics (No. 6142A05230501), National Science Challenge Project (TZ2025007), National Center for Mathematics and Interdisciplinary Science, CAS.}

% ==================== author ====================
\author[1]{Chenhao Jin}[orcid=0009-0001-5857-5400]
% \cormark[1]

\ead{kanit@mail.ustc.edu.cn}

\affiliation[1]{organization={School of Mathematical Sciences, University of Science and Technology of China},city={Hefei},state={Anhui},postcode={230026},country={China}}

\author[2,3]{Hehu Xie}[orcid=0000-0002-1947-9256]
\cormark[1]

\ead{hhxie@lsec.cc.ac.cn}

\affiliation[2]{organization={SKLMS, Academy of Mathematics and Systems Science, Chinese Academy of Sciences},addressline={No.55, Zhongguancun Donglu}, city={Beijing},postcode={100190}, country={China}}

\affiliation[3]{organization={School of Mathematical Sciences, University of Chinese Academy of Sciences},
city={Beijing}, postcode={100049},country={China}}

\cortext[1]{Corresponding author}

% abstract
\begin{abstract}
This paper develops a robust solver for the Maxwell eigenproblem in 3D photonic crystals with anisotropic media. The solver employs the kernel compensation technique under the framework of Yee's scheme to eliminate null space and enable matrix-free, GPU-accelerated operations via 3D discrete Fourier transform. Furthermore, we propose a novel discretization for permittivity tensor containing off-diagonal entries and prove that the resulting matrix is Hermitian positive definite, which ensures the correctness of the kernel compensation technique. Numerical experiments on several benchmark examples are demonstrated to validate the robustness and accuracy of our scheme.
\end{abstract}

% highlights
%\begin{highlights}
%\item We develop a kernel compensation method that eliminates null space based on Yee's scheme and provide precise characterization of the null space and the selection of the penalty coefficient.
%\item We propose novel discretizations for dielectric coefficient in anisotropic media and prove the Hermitian positive definiteness of the resulting discrete dielectric matrices.
%\item We derive matrix-free operations via 3D discrete Fourier transform and implement GPU acceleration on the whole process for high performance of large-scale simulations.
%\end{highlights}

% keywords seperated by \sep

\maketitle

\section{Introduction}\label{se:intro}
Photonic crystals (PCs) are periodic dielectric or metallic nanostructures recognized as a cornerstone of modern photonics \cite{joannopoulos2008book}. By offering a mechanism to create photonic bandgaps, PCs provide special capabilities to mold the flow of light, which is essential for developing integrated devices such as efficient spectral filters, on-chip lasers, and optical transistors. The mathematical challenge lies in effectively solving the time-harmonic Maxwell's equations for dealing with null space and numerical accuracy. 
The governing system, defined in $\mathbb{R}^3$, starts with:
\begin{align}\label{display:intro_maxwell_magnetic}
\left\{
\begin{aligned}	
&\nabla\times(\varepsilon^{-1}\nabla\vec\times \vec H)=\omega^2\vec H,\quad&  \text{in  }\Rbb^3,\\
&\nabla\cdot\vec H=0, \quad& \text{in  }\Rbb^3.
\end{aligned}
\right.				
\end{align}
In this eigenvalue problem, $\omega$ is the frequency, $\vec H$ is the magnetic field and $\varepsilon$ is the normalized dielectric coefficient. According to Bloch's theorem, $\vec H$ satisfies the quasi-periodic condition 
\begin{align}\label{display:intro_quasiperiodic}
\vec H(\vec x+\vec a_i)=e^{\imagunit\wnk\cdot\vec a_i}\vec H(\vec x),\quad i=1,2,3, 
\end{align}
where $\wnk=(\wnkk_1,\wnkk_2,\wnkk_3)\in\Rbb^3$ is the wave number vector located in the first Brillouin zone, 
$\left\{\vec a_i\right\}_{i=1}^3$ are lattice translation vectors that span the primitive cell
\begin{align}\label{display:intro_domain}
\Omega:=\left\lbrace \vec x \in \Rbb^3: \vec x = \sum_{i=1}^3 x_i  \vec a_i,
\; x_i \in [0,1], \;i=1, 2, 3 \right\rbrace.
\end{align} 
One can find detailed definitions of $\{\vec a_i\}$ and $\wnk$ in \eqref{display:exp_latticeconst}, \eqref{display:exp_bzsym} in Section \ref{se:num}. We introduce the transformation $\vec H_{\wnk}(\vec x):=e^{-\imagunit\wnk\cdot\vec x}\vec H(\vec x)$ to convert the problem into the one with pure periodic boundary conditions involving the shifted nabla operator $\nabla_{\wnk}:=\nabla+\imagunit\wnk I$. To avoid ambiguity of notation, the subscript ``$\wnk$'' of ``$\vec H_{\wnk}$'' is left out. We keep using ``$\vec H$'' to refer to the magnetic field and its variants hereafter:
\begin{align}\label{display:intro_magneticshifted}
\left\{
\begin{aligned}					
&\nabla_{\wnk}\times(\varepsilon^{-1}\nabla_{\wnk}\times\vec  H)
=\omega^2\vec  H,\quad& \text{in }\Omega,\\
&\nabla_{\wnk}\cdot\vec  H=\vec 0, \quad& \text{in } \Omega, \\
&  \vec  H(\vec x+\vec a_i)=\vec  H(\vec x), \quad&i=1,2,3.
\end{aligned}
\right.
\end{align}
This transformation is common in variational approaches \cite{lu2022parallelpc} but is particularly advantageous for finite difference schemes, as the resulting simpler cyclic structure is crucial for 
implementing matrix-free operations.

The inverse permittivity $\varepsilon^{-1}$ is a piecewise constant defined by
\begin{align}\label{display:intro_permittivity}
\varepsilon^{-1}(\vec x)=\begin{cases}\varepsilon_1,\quad \vec x\in\Omega_1,\\ I_3,\quad \vec x\in\Omega_0:=\Omega\backslash\Omega_1,\end{cases}
\end{align}
where $\varepsilon_1$ is a $3\times 3$ Hermitian positive definite (HPD) matrix. The media of the system \eqref{display:intro_magneticshifted} is classified as isotropic when $\varepsilon_1$ is a scalar matrix, anisotropic when $\varepsilon_1$ is a diagonal matrix with distinct entries or $\varepsilon_1$ contains nonzero off-diagonal entries. This paper focuses on the behavior of the system with anisotropic media.

We present the kernel compensation formulation in our prior work \cite{jin2025kcpc} to handle the huge null space brought by $\ker(\nabla_{\wnk}\times)$ as well as the divergence-free condition $\nabla_{\wnk}\cdot\vec H=\vec 0$, which yields an HPD system with an efficient preconditioner. Our present focus is a novel discretization for $\varepsilon^{-1}$ in \eqref{display:intro_permittivity} where nonzero entries of $\varepsilon_1$ occur. Unlike finite element methods \cite{chou2019fepc, dobson1999pc2d, lu2017dg} or plane wave expansion methods \cite{ho1990pwe,sailor1998pwe} which yield 
generalized eigenvalue problems, the finite difference framework, such as \cite{huang2013eigendecomppc,huang2015null}, produces simpler eigenvalue problems and benefits from circulant matrix structures. Additionally, by applying a coordinate transform to the boundary conditions in \eqref{display:intro_magneticshifted}, we align the periodicity with the standard $\Rbb^3$ basis, enabling highly efficient GPU acceleration.

In this paper, we overcome the limitations of 
solving the anisotropic system with finite difference method 
and present a corresponding robust, efficient eigensolver built upon the kernel compensation technique. We summarize the key innovations 
as follows:
\begin{itemize}
\item We complete the theoretical foundation of the kernel compensation technique by rigorously demonstrating the exact structure of the null space and a detailed characterization of the penalty coefficient. 

\item By using 3D discrete Fourier transform (DFT), we design matrix-free discrete differential operators, reducing the per-iteration cost of the eigensolver to a single 3D global DFT. Furthermore, we extend the Yee's scheme to high-order discretizations under the framework of matrix-free operations, which improves the accuracy under certain cases.
\item We propose an innovative discretization concerning permittivity $\varepsilon$ in anisotropic media, proving that the corresponding discrete matrix is HPD under physical constraints. This HPD property is essential for ensuring the correctness of the kernel compensation technique.
\item We present extensive numerical computations on various anisotropic systems and implement GPU acceleration on the whole process for high performance of large-scale simulations.
\end{itemize}

The paper is outlined as follows. Section \ref{se:mfd} formally introduces the finite difference discretization process and the kernel compensation formulation. Section \ref{se:permittivity} details our innovative treatment of the anisotropic permittivity tensor and establishes the theoretical proof of the HPD property. Section \ref{se:num} presents extensive numerical experiments, including band structure analysis and computational performance benchmarks. We conclude in Section \ref{se:con}.

\section{Discretization and Kernel Compensation}	\label{se:mfd}
In Subsection \ref{sse:yee}, we present Yee's scheme \cite{yee1966} to discretize Maxwell's equation with shifted nabla operator. Matrix representations of discrete differential operators are also provided. The results are generalized to arbitrary lattices $\left\{\vec a_i\right\}_{i=1}^3$ through coordinate transform, which is covered in Subsection \ref{sse:ct}. We then discuss the kernel compensation technique that eliminates the null space in Subsection \ref{sse:kc}. In Subsection \ref{sse:matfree}, we imply matrix-free operations to replace sparse matrix-vector multiplication and design an efficient preconditioner. In Subsection \ref{sse:kc:highord}, we introduce high-order finite difference stencils and extend previous results to high-order discretizations.

\subsection{Yee's Scheme}\label{sse:yee}	
In this subsection, we set the primitive lattice vectors to be the standard orthonormal basis of $\Rbb^3$, i.e., $\vec a_i = \vec e_i$ and $\Omega = (0,1)^3$. We uniformly partition the domain $\Omega$ into a cubic grid with division $N\in\Nbb^*$, grid size $h=1/N$ in each $x, y, z$ direction. Although the generalization to distinct divisions along each axis is straightforward, it does not yield improvements. The distributions of node, edge, face and volume DoFs are shown in Figure \ref{fig:DoFs}.
\begin{figure}
\centering
%\raggedright
%\hspace*{-1.3cm}
\includegraphics[width=\linewidth]{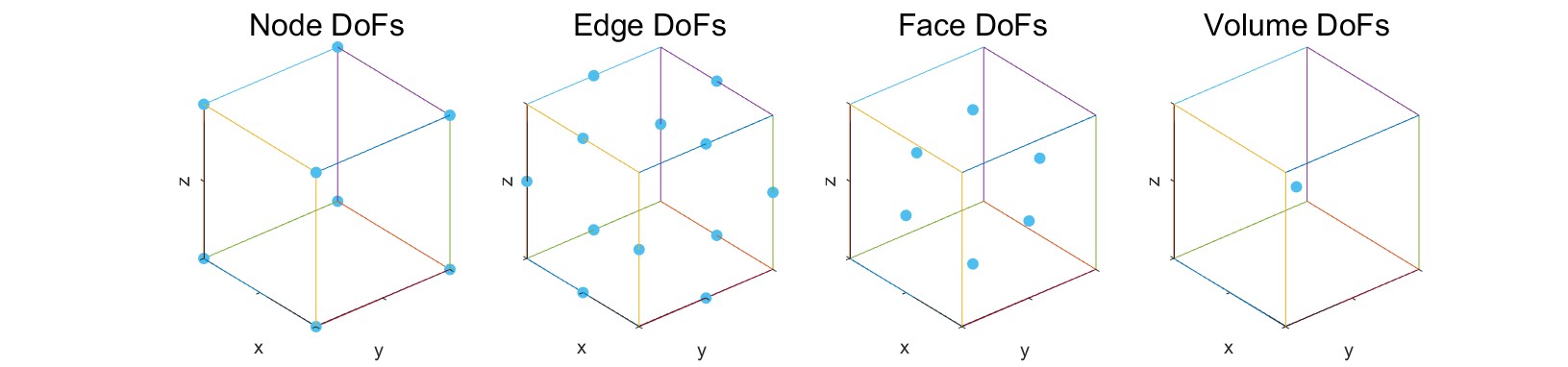}  
\caption{DoFs of scalar and vector grid functions on a single cell}
\label{fig:DoFs}
\end{figure}

In Yee's staggered grid \cite{yee1966} the magnetic field $\vec H=(H_1,H_2,H_3)$ is approximated by face DoFs, while electric field $\vec E=(E_1,E_2,E_3)$ by edge DoFs:
\begin{align}\label{display:yee_edgefacedofs}
\begin{aligned}
&\vec E\mapsto \vec E_h:=\begin{pmatrix}
\vec E^1_h, & \vec E^2_h, & \vec E^3_h
\end{pmatrix}^\top :=\begin{pmatrix}
E^1_{i-\frac{1}{2},j,k},&E^2_{i,j-\frac{1}{2},k},&
E^3_{i,j,k-\frac{1}{2}}
\end{pmatrix}^\top,\\
&\vec H\mapsto \vec H_h:=\begin{pmatrix}
\vec H^1_h,&\vec H^2_h,&\vec H^3_h
\end{pmatrix}^\top :=\begin{pmatrix}
H^1_{i,j-\frac{1}{2},k-\frac{1}{2}},&H^2_{i-\frac{1}{2},j,k-\frac{1}{2}},&
H^3_{i-\frac{1}{2},j-\frac{1}{2},k}
\end{pmatrix}^\top,
\end{aligned}
\end{align}
where $1\leq i,j,k\leq N$. Subscripts represent discrete coordinates on the Cartesian grid where the function is evaluated. The image of the divergence operator $\nabla_{\wnk}\cdot$ acting on an arbitrary periodic scalar function $\psi$, for example, is approximated by volume DoFs:
\begin{align}\label{display:yee_volumedofs}
\psi\mapsto\vec\psi_h:=\begin{pmatrix}
\psi_{i-\frac{1}{2},j-\frac{1}{2},k-\frac{1}{2}}
\end{pmatrix}.
\end{align}
Therefore we introduce discrete spaces of edge, face and volume DoFs as
\begin{align}\label{display:yee_dofspace}
\begin{aligned}
&\Ecal_h:=\left\{\begin{pmatrix}
E^1_{i-\frac{1}{2},j,k}, & E^2_{i,j-\frac{1}{2},k}, &
E^3_{i,j,k-\frac{1}{2}}
\end{pmatrix}: 
\vec E\text{ is a periodic field on }\Omega\right\}\simeq\Cbb^{3N^3},\\
&\Fcal_h:=\left\{\begin{pmatrix}
H^1_{i,j-\frac{1}{2},k-\frac{1}{2}}, & H^2_{i-\frac{1}{2},j,k-\frac{1}{2}},&
H^3_{i-\frac{1}{2},j-\frac{1}{2},k}
\end{pmatrix}:\vec H\text{ is a periodic field on }\Omega\right\}\simeq\Cbb^{3N^3},\\
&\Vcal_h:=\left\{\begin{pmatrix}
\psi_{i-\frac{1}{2},j-\frac{1}{2},k-\frac{1}{2}}
\end{pmatrix}:\psi\text{ is a periodic scalar function on }\Omega\right\}\simeq \Cbb^{N^3}.
\end{aligned}
\end{align}

In the double curl problem \eqref{display:intro_magneticshifted}, 
the discrete outer curl operator maps edge DoFs to face DoFs. 
Crucially, the inner curl operator is derived as the adjoint of 
the outer operator, mapping face DoFs to edge DoFs. We denote the matrix representations for the outer curl operator $\nabla_{\wnk}\times$, 
the permittivity $\varepsilon$ and the divergence operator $\nabla_{\wnk}\cdot$ as $\Acal$, $\Mcal_\varepsilon$ and $\Bcal$, respectively. Then due to the adjoint relation \cite{jin2025kcpc,lipnikov2014mfd}, matrix representation of the inner curl operator $\nabla_{\wnk}\times$ is simply $\Acal^\conjT$, where ``$\conjT$'' 
denotes the conjugate transpose. We have
\begin{align}\label{display:yee_dofmap}
\Acal:\Ecal_h\to\Fcal_h,\quad\Mcal_\varepsilon:\Ecal_h\to\Ecal_h,\quad \Acal^\dagger:\Fcal_h\to\Ecal_h,\quad\Bcal:\Fcal_h\to\Vcal_h.
\end{align}

For any  $\vec E_h\in\Ecal_h$, the image $\Acal\vec E_h$ along each axis has the following expressions:
\begin{eqnarray*}\label{display:yee_imagecurl}
%\begin{split}
(\Acal\vec E_h)_{i,j-\frac{1}{2},k-\frac{1}{2}}&=&\frac{E^3_{i,j,k-\frac{1}{2}}-E^3_{i,j-1,k-\frac{1}{2}}}{h}-\frac{E^2_{i,j-\frac{1}{2},k}-E^2_{i,j-\frac{1}{2},k-1}}{h}\\
&&+\imagunit\left(\wnkk_2\frac{E^3_{i,j-1,k-\frac{1}{2}}+E^3_{i,j,k-\frac{1}{2}}}{2}-\wnkk_3\frac{E^2_{i,j-\frac{1}{2},k-1}+E^2_{i,j-\frac{1}{2},k}}{2}\right),\\
(\Acal\vec E_h)_{i-\frac{1}{2},j,k-\frac{1}{2}}&=&\frac{E^1_{i-\frac{1}{2},j,k}-E^1_{i-\frac{1}{2},j,k-1}}{h}-\frac{E^3_{i,j,k-\frac{1}{2}}-E^3_{i-1,j,k-\frac{1}{2}}}{h}\\
&&+\imagunit\left(\wnkk_3\frac{E^1_{i-\frac{1}{2},j,k-1}+E^1_{i-\frac{1}{2},j,k}}{2}-\wnkk_1\frac{E^3_{i-1,j,k-\frac{1}{2}}-E^3_{i,j,k-\frac{1}{2}}}{2}\right),\\
(\Acal\vec E_h)_{i-\frac{1}{2},j-\frac{1}{2},k}&=&\frac{E^2_{i,j-\frac{1}{2},k}-E^2_{i-1,j-\frac{1}{2},k}}{h}-\frac{E^1_{i-\frac{1}{2},j,k}-E^1_{i-\frac{1}{2},j-1,k}}{h}\\
&&+\imagunit\left(\wnkk_1\frac{E^2_{i-1,j-\frac{1}{2},k}+E^2_{i,j-\frac{1}{2},k}}{2}-\wnkk_2\frac{E^1_{i-\frac{1}{2},j-1,k}+E^1_{i-\frac{1}{2},j,k}}{2}\right).
%\end{split}
\end{eqnarray*}
Likewise, for any $\vec H_h\in\Fcal_h$, the discrete divergence is formulated by
\begin{eqnarray*}\label{display:yee_imagediv}
(\Bcal\vec H_h)_{i-\frac{1}{2},j-\frac{1}{2},k-\frac{1}{2}}
&=&\frac{H^1_{i,j-\frac{1}{2},k-\frac{1}{2}}
-H^1_{i-1,j-\frac{1}{2},k-\frac{1}{2}}}{h}
+\frac{H^2_{i-\frac{1}{2},j,k-\frac{1}{2}}-H^2_{i-\frac{1}{2},j-1,k-\frac{1}{2}}}{h}\\
&&+\frac{H^3_{i-\frac{1}{2},j-\frac{1}{2},k}-H^3_{i-\frac{1}{2},j-\frac{1}{2},k-1}}{h}+\imagunit\left(\wnkk_1\frac{H^1_{i-1,j-\frac{1}{2},k-\frac{1}{2}}
+H^1_{i,j-\frac{1}{2},k-\frac{1}{2}}}{2}\right.\\
&&\left.+\wnkk_2\frac{H^2_{i-\frac{1}{2},j-1,k-\frac{1}{2}}
+H^2_{i-\frac{1}{2},j,k-\frac{1}{2}}}{2}
+\wnkk_3\frac{H^3_{i-\frac{1}{2},j-\frac{1}{2},k-1}
+H^3_{i-\frac{1}{2},j-\frac{1}{2},k}}{2}\right).
\end{eqnarray*}
The notations of $\vec E_h$ and $\vec H_h$ are adopted from \eqref{display:yee_dofspace}. 
We may formulate the explicit matrix representations through circulant matrix below:
\begin{align}\label{display:yee:D0D1}
\Dbf_1:=\frac{1}{h}\begin{pmatrix}
1 & \ & \ & -1 \\ -1 & 1 & \ & \ \\
\ & \ & \ddots & \ \\
\ & \ & -1  & 1
\end{pmatrix}\in\Rbb^{N\times N},\quad 
\Dbf_0:=\frac{1}{2}\begin{pmatrix}
1 & \ & \ & 1 \\ 1 & 1 & \ & \ \\
\ & \ & \ddots & \ \\
\ & \ & 1  & 1
\end{pmatrix}\in\Rbb^{N\times N}.
\end{align}
Then directional derivatives combined with actions of $\wnk$ can be represented by the following matrix blocks using Kronecker product $\otimes$ ($I_N$ denotes the identity matrix of dimension $N$):
\begin{align}\label{display:yee_block}
\begin{aligned}
&\Dcal_i=\Dcal_{1,i}+\imagunit\wnkk_i\Dcal_{0,i}\quad i=1,2,3,\text{ where}\\
&\Dcal_{s,1}=I_{N^2}\otimes\Dbf_s,\quad \Dcal_{s,2}=I_N\otimes\Dbf_s\otimes I_N,\quad \Dcal_{s,3}=\Dbf_s\otimes I_{N^2},\quad s=1,2.
\end{aligned}
\end{align}
Therefore the matrix representations of $\Acal$ and $\Bcal$ are given by
\begin{align}\label{display:yee_ABmatrix}
\Acal=\begin{pmatrix}
\ & -\Dcal_3 & \Dcal_2\\ \Dcal_3 & \ & -\Dcal_1\\ -\Dcal_2 & \Dcal_1 &\
\end{pmatrix},\quad \Bcal=\begin{pmatrix}
\Dcal_1 & \Dcal_2 & \Dcal_3
\end{pmatrix}.		
\end{align}
The assembling of the discrete permittivity matrix $\Mcal_\varepsilon$ will be covered 
in Section \ref{se:permittivity}. In the rest of this section 
we assume that $\Mcal_{\varepsilon}$ is HPD.

\subsection{Coordinate Transform}\label{sse:ct}
In Subsection \ref{sse:yee}, the periodic cell $\Omega$ is assumed to be a unit cube. 
For general cases, the formulations \eqref{display:yee_block} 
and \eqref{display:yee_ABmatrix} can be generalized for any 
linearly independent lattice vectors $\{\vec a_i\}_{i=1}^3$ through a simple coordinate transform. If we apply the following operation:
\begin{align}\label{display:ct_replacement}
\text{Replace } \vec H\text{ in \eqref{display:intro_magneticshifted} by }
\left(\vec x\mapsto\vec H(\Abf\vec x)\right),\quad\Abf=(\vec a_1,\vec a_2,\vec a_3),
\end{align}
then equation \eqref{display:intro_magneticshifted} can be equivalently written as
\begin{align}\label{display:ct_magneticCT}
\left\{
\begin{aligned}
&\nabla_{\Abf,\wnk}\times\varepsilon^{-1}\nabla_{\Abf,\wnk}\times\vec H
=\omega^2\vec H,\quad \text{in }(0,1)^3,\\
&\nabla_{\Abf,\wnk}\cdot\vec H=0,\quad \text{in } (0, 1)^3,\\
& \vec H(\vec x+\vec e_i)=H (\vec x), \quad i=1,2,3,
\end{aligned}
\right.
\end{align}
where $\nabla_{\Abf,\wnk}=\nabla_{\Abf}+\imagunit\wnk I$ and $\nabla_{\Abf}$ is given by
\begin{align}\label{display:ct_newnabla}
\nabla_{\Abf}=\left(\sum\limits_{j=1}^3b_{j1}\frac{\p}{\p x_j}\ ,\ \sum\limits_{j=1}^3b_{j2}\frac{\p}{\p x_j}\ ,\ \sum\limits_{j=1}^3b_{j3}\frac{\p}{\p x_j}\right)^\top ,\quad b_{ij}=(\Abf^{-1})_{ij}.
\end{align}
In this case we reformulate the matrix blocks in \eqref{display:yee_block} as
\begin{align}\label{display:ct_blocks}
\begin{aligned}
\hat\Dcal_i=\sum\limits_{j=1}^3b_{ji}\Dcal_{1,i}+\imagunit\wnkk_i\Dcal_{0,i},\quad i=1,2,3.
\end{aligned}
\end{align}
The matrices $\Acal$ and $\Bcal$ preserve the structure in \eqref{display:yee_ABmatrix} by updating matrix blocks $\hat\Dcal_i$. In summary, the discrete formulation of model problem \eqref{display:intro_magneticshifted} as well as \eqref{display:ct_magneticCT} results in the system:
\begin{align}\label{display:ct_preliminaryresult}
\left\{
\begin{aligned}
&\left(\Acal\Mcal_\varepsilon\Acal^\conjT\right)\vec H_h=\omega_h^2\vec H_h,\\
&\Bcal\vec H_h=\vec 0.
\end{aligned}
\right.
\end{align}

\subsection{Kernel Compensation}\label{sse:kc}
Motivated by the mixed formulation of Maxwell's equations \cite{kikuchi1989mixedform}, kernel compensation is introduced to find the nontrivial eigenpairs that satisfy the discrete divergence-free condition $\Bcal\vec H_h=0$ while avoiding the null space of $\Acal^\dagger$, which is formulated as:
\begin{align}\label{display:kc_formulation}
\begin{aligned}
&\textbf{The first several smallest positive eigenvalues of}\\
&(\Acal\Mcal_\varepsilon\Acal^\conjT+\gamma\Bcal^\conjT\Bcal)\vec H_h=\omega_h^2\vec H_h\\
&\textbf{coincide with \eqref{display:ct_preliminaryresult} when }\gamma>0\textbf{ is sufficiently large.}
\end{aligned}
\end{align}
Now, we come to prove the correctness of \eqref{display:kc_formulation} with an explicit choice of the penalty coefficient $\gamma$. 
For this aim, we first introduce the following lemma. 
%The following proposition guarantees the
\begin{lem}\label{lem:kc:diagofcirc}
If $\Cbf\in \Cbb^{N\times N}$ is a circulant matrix with the first row $(c_1,\cdots, c_N)$, 
then it can be diagonalized by the unitary DFT matrix $\Fbf$ as
\begin{align}\label{display:kc_diagofcirc}
\begin{aligned}
&\Cbf=\Fbf\diag(\lambda_1,\cdots,\lambda_N)\Fbf^\conjT,
\quad \lambda_i=\sum\limits_{j=1}^{N}c_j\omega^{(i-1)(j-1)},\quad i=1,\cdots,N.
\end{aligned}
\end{align}
The matrix $\Fbf=(F_{ij})_{1\leq i,j\leq N}$ is defined by $ F_{ij}=\frac{1}{\sqrt{N}}\omega^{(i-1)(j-1)}$ and $\omega=\exp\left(\imagunit\frac{2\pi}{N}\right)$. As a corollary, the product of any two circulant matrices is commutative: $\Cbf_1\Cbf_2=\Cbf_2\Cbf_1$ whenever $\Cbf_1$ and $\Cbf_2$ are circulant matrices.
\end{lem}
\begin{proof}
The proof is provided in \cite{chen1987solcirc}.
\end{proof}
\begin{prop}\label{prop:kc_BAzero}
The matrices $\Acal$ and $\Bcal$ defined in \eqref{display:yee_block}, \eqref{display:yee_ABmatrix} 
and \eqref{display:ct_blocks} satisfy: $\Bcal\Acal=\vec 0$.
\end{prop}
\begin{proof}%[Proof of Proposition \ref{prop:kc_BAzero}]
It suffices to verify that $\hat\Dcal_i\hat\Dcal_j=\hat\Dcal_j\hat\Dcal_i$, which is equivalent 
to the commutativity of the product of $\Dcal_{s,i}$ and 
$\Dcal_{t,j}$ for any $s,t\in\{0,1\}$ and $i,j\in\{1,2,3\}$. When $i\neq j$ 
the result is proved by the mixed product property \begin{align}\label{display:kc_mixedproductprop}
(A_1\otimes B_1)(A_2\otimes B_2)=(A_1A_2)\otimes(B_1B_2)
\end{align}for any $A_i$, $B_i$ with compatible shapes. Thus multiplications of $\Dcal_{s,i}$ 
and $\Dcal_{t,i}$ are commutative. When $i=j$ the result holds due to Lemma \ref{lem:kc:diagofcirc} 
and the mixed product property \eqref{display:kc_mixedproductprop}.
\end{proof}

\begin{prop}[Main Result of Kernel Compensation]\label{prop:kc_mainprop}
Let $\Acal,\Mcal\in\Cbb^{N^\Acal\times N^\Acal}$ and $\Bcal\in \Cbb^{N^\Bcal\times N^\Acal}$be matrices such that: $\Bcal\Acal=\vec 0$, and $\Mcal$ is HPD. We assume that 
\begin{align}\label{display:kc_eigendistributionofAB}
\begin{aligned}
&\textbf{eigenvalues of }\Acal \Mcal\Acal^\conjT:\quad\overbrace{0,0,\cdots,0}^{N^\Acal_{0}\text{ zeros}}<\overbrace{\lambda^\Acal_{1}\leq\lambda^\Acal_{2}\leq\cdots\leq\lambda^\Acal_{N^\Acal_{1}}}^{N_{1}^\Acal\text{ nonzeros}},\\
&\textbf{eigenvalues of }\Bcal^\conjT\Bcal:\quad\overbrace{0,0,\cdots,0}^{N^\Bcal_{0}\text{ zeros}}<\overbrace{\lambda^\Bcal_{1}\leq\lambda^\Bcal_{2}\leq\cdots\leq\lambda_{N^\Bcal_{1}}^\Bcal}^{N^\Bcal_{1}\text{ nonzeros}}.
\end{aligned}
\end{align}
Then the eigenvalue distributions of $\Acal\Mcal\Acal^\conjT+\gamma\Bcal^\conjT\Bcal$ are
\begin{align}
\overbrace{0,\cdots,0}^{\dim\Hcal\text{ zeros}}<\lambda_1\leq\lambda_2\leq\cdots\leq\lambda_m\leq\cdots,
\end{align}
where $\gamma$ is an arbitrary positive number, $\Hcal:=\ker\Acal^\conjT\cap\ker\Bcal$. There holds
\begin{align}
\left\{\lambda_1,\cdots,\lambda_m,\cdots\right\}=\left\{\lambda_1^\Acal,\cdots,\lambda_{N_1^\Acal}^\Acal\right\}
\cup\left\{\gamma\lambda_1^\Bcal,\cdots,\gamma\lambda_{N_1^\Bcal}^\Bcal\right\}.
\end{align}
\end{prop}
\begin{proof}
We refer the reader to \cite{jin2025kcpc} for the proof.
\end{proof}

The kernel compensation formulation \eqref{display:kc_formulation} allows us to recover the nonzero eigenvalues of $\Acal\Mcal_\varepsilon\Acal^\conjT$ by computing the spectrum of the penalized system $\Acal\Mcal_\varepsilon\Acal^\conjT+\gamma\Bcal^\conjT\Bcal$. Proposition \ref{prop:kc_mainprop} 
shows the first $m$ smallest positive eigenvalues of the two systems are identical, provided that $\gamma$ satisfies $\gamma\lambda_1^\Bcal\geq\lambda_m^\Acal$. To characterize the resulting null space $\Hcal$ 
and determine an appropriate value for the penalty coefficient $\gamma$, we present the following analysis.

\begin{lem}\label{lem:kc_ABblock}
When the coordinate change \eqref{display:ct_replacement} $\Abf=I_3$, we have

(1). The block matrices $\Acal$ and $\Bcal$ satisfy 
$\Acal\Acal^\conjT+\Bcal^\conjT\Bcal=\diag(\Lcal,\Lcal,\Lcal)$ with 
\begin{align}\label{display:kc_Lmat}
\Lcal=\Bcal\Bcal^\conjT=\Dcal_1^\conjT\Dcal_1+\Dcal_2^\conjT\Dcal_2+\Dcal_3^\conjT\Dcal_3\in\Cbb^{N^3\times N^3},
\end{align}

(2). If there holds conditions
\begin{align}\label{display:kc_condition_kneq0}
\wnk\neq\vec 0,\quad\max\limits_{i=1,2,3}|\wnkk_i|\leq\pi,\quad \text{grid size }h=\frac{1}{N}<\frac{1}{\pi},
\end{align}
then the smallest nonzero eigenvalue of $\Bcal^\conjT\Bcal$, as is denoted by $\lambda_1^\Bcal$ in Proposition \ref{prop:kc_mainprop}, satisfies $\lambda_1^\Bcal=\lambda_{\min}(\Lcal)=\normtwo{\wnk}^2$.
\end{lem}
\begin{proof}
A direct computation of $\Acal\Acal^\conjT+\Bcal^\conjT\Bcal$ yields a block diagonal matrix, provided that $\Dcal_i^\conjT\Dcal_j=\Dcal_j\Dcal_i^\conjT$ for $i\neq j$. This commutativity is 
guaranteed by the mixed product property of Kronecker products \eqref{display:kc_mixedproductprop}. The identity $\Dcal_i^\conjT\Dcal_i=\Dcal_i\Dcal_i^\conjT$ holds because $\Kbf_i:=\Dbf_1+\imagunit\wnkk_i\Dbf_0$ is a circulant matrix and thus normal, 
so is each block $\Dcal_i$; We establish \eqref{display:kc_Lmat}.

To prove (2), since $\Bcal^\conjT\Bcal$ shares the same nonzero eigenvalues 
with $\Lcal=\Bcal\Bcal^\conjT$, we analyze the spectrum of $\Lcal$. 
The matrix $\Kbf_i^\conjT\Kbf_i$ is circulant, 
and by setting $c_i=\frac{1}{h}+\imagunit\frac{\wnkk_i}{2}$, 
its first row is $$(2|c_i|^2,-c_i^2,0,\cdots,0,-\overline{c_i}^2).$$
Applying Lemma \ref{lem:kc:diagofcirc}, its eigenvalues are
\begin{align}\label{display:kc_lapblockeigs}
\mu^{(i)}_j:=\lambda_j(\Kbf_i^\conjT\Kbf_i)=2\left(|c_i|^2-\text{Re}(c_i^2\omega^{j})\right)
=\left(\frac{2}{h^2}+\frac{\wnkk_i^2}{2}\right)\left(1-\cos(\phi_i+2j\pi h)\right),
\end{align}
for $j=1,\cdots N$ and $\tan\phi_i=\frac{4\wnkk_ih}{4-\wnkk_i^2h^2}$. The spectrum of $\Lcal=I_{N^2}\otimes(\Kbf_1^\conjT\Kbf_1)+I_N\otimes(\Kbf_2^\conjT\Kbf_2)\otimes I_N
+(\Kbf_3^\conjT\Kbf_3)\otimes I_{N^2}$ is then presented by
\begin{align}\label{display:kc_Leig}
\left\{\mu^{(1)}_{j_1}+\mu^{(2)}_{j_2}+\mu^{(3)}_{j_3}:1\leq j_1,j_2,j_3\leq N\right\}.
\end{align}

We then aim to find the minimum of $\mu_j^{(i)}$, which is $\min\{\mu_{N-1}^{(i)},\mu_N^{(i)}\}$. We define $f(x)=\arctan\left(\frac{4x}{4-x^2}\right)$, where $x=\wnkk_ih$ is such that $|x|<1$ due to the condition \eqref{display:kc_condition_kneq0}. By the mean value theorem, $|f(x)|=|f(x)-f(0)|=|f'(c)x|$ for some $c$ strictly between $0$ and $x$. Since $f'(c)=\frac{4}{4+c^2}\in(0,1)$, we have $|\phi_i|=|f(x)|<|x|=|\wnkk_ih|.$ Thus $|\phi_i|<\pi h$ and
\begin{align}\label{display:kc_minKeigen}
\begin{aligned}
\min\limits_{1\leq j\leq N}\mu_j^{(i)}&
=\left(\frac{2}{h^2}+\frac{\wnkk_i^2}{2}\right)\left(1-\max\{\cos(\phi_i),\cos(2\pi h-\phi_i)\}\right)=\wnkk_i^2.
\end{aligned}
\end{align}
Here we substitute $\cos(\phi_i)=\frac{4-\wnkk_i^2h^2}{4+\wnkk_i^2h^2}$ 
derived from its tangent value into the above computation. 
Finally, due to \eqref{display:kc_Leig}, the minimum eigenvalue of $\Lcal$ is $\normtwo{\wnk}^2$, and thus $\lambda_1^\Bcal=\normtwo{\wnk}^2$.
\end{proof}

\begin{prop}\label{prop:kc_finalresult}
For the kernel compensation formulation \eqref{display:kc_formulation} with coordinate change $\Abf=I_3$, the null space $\Hcal$ and the choice of the penalty coefficient $\gamma$ are determined as follows:

(1). The null space $\Hcal$ is defined by: If $\wnk=\vec 0$, then
\begin{align}\label{display:kc_keq0}
\begin{aligned}
\Hcal=\spann\{\vec o_i\}_{i=1}^3,\quad\vec o_i\in\Rbb^{3N^3},\quad \vec (\vec o_i)_j=\begin{cases}
1,\quad\text{if }(i-1)N^3+1\leq j\leq iN^3,\\
0,\quad\text{otherwise.}
\end{cases}
\end{aligned}
\end{align}
If $\wnk\neq\vec0$ and condition \eqref{display:kc_condition_kneq0} holds, then $\Hcal=\{\vec0\}$.

(2). The following choice of $\gamma$ is sufficient for the correctness of \eqref{display:kc_formulation}:
\begin{align}\label{display:kc_settingofgamma}
\gamma=\begin{cases}
\max\left\{1,\lambda_m^\Acal\right\},\quad &\wnk=\vec 0,\\
\max\left\{1,\lambda_m^\Acal\right\}\cdot\max\left\{1,\normtwo{\wnk}^{-2}\right\},\quad &\wnk\neq\vec 0.
\end{cases}
\end{align}
\end{prop}
\begin{proof}
If $\wnk\neq\vec 0$, Lemma \ref{lem:kc_ABblock} shows that $\Lcal$ is invertible, 
$\ker\Lcal=\{0\}$. It follows that $$\Hcal=\ker(\Acal\Acal^\conjT+\Bcal^\conjT\Bcal)=\{\vec 0\}.$$
If $\wnk=\vec 0$, $\Lcal=\sum\limits_{i=1}^3\Dcal_i^\conjT\Dcal_i$, 
we have  $\ker\Lcal=\bigcap\limits_{i=1}^3\ker\Dcal_i$ due to the positive 
semi-definiteness. The circulant matrix $\Dbf_1$ has a 1D kernel 
spanned by the vector of all ones, $\vec 1_N$. Due to the Kronecker product 
structure of the $\Dcal_i$, their common kernel is $\ker\Lcal = \spann\{\vec 1_{N^3}\}$ 
with $\vec 1_{N^3}=\vec 1_N\otimes\vec 1_N\otimes\vec1_N$. It follows that $\Hcal$ 
is spanned by $(\vec 1_{N^3}, \vec 0, \vec 0)^\top $, $(\vec 0, \vec 1_{N^3}, \vec 0)^\top $, 
and $(\vec 0, \vec 0, \vec 1_{N^3})^\top $, which are the vectors $\{\vec o_i\}_{i=1}^3$.

To ensure $\gamma\lambda_1^\Bcal\geq\lambda_m^\Acal$: if $\wnk\neq\vec 0$, Proposition \ref{lem:kc_ABblock} reveals $\lambda_1^\Bcal=\normtwo{\wnk}^2$, the condition is satisfied; if $\wnk=\vec 0$, the smallest nonzero eigenvalue of $\Kbf_i^\conjT\Kbf_i$ equals $\frac{4}{h^2}\sin^2(\pi h)$. Since $\frac{\sin x}{x}>\sin1$ within the bound $x=\pi h\in(0,1)$, there holds $\lambda_1^\Bcal>4\pi^2(\sin1)^2$. Thus $\gamma=\max\{1,\lambda_m^\Acal\}$ is sufficient.
\end{proof}

\begin{rmk}
Condition \eqref{display:kc_condition_kneq0} is reasonable: $|\wnkk_i|\leq\pi$ is satisfied 
by construction at high symmetry points, as is shown in \eqref{display:exp_bzsym} 
in Section \ref{se:num}; $N>\pi$ holds for any reasonable grid. Furthermore, 
in our numerical experiments the desired eigenvalues are such that $\lambda_m^\Acal<4\pi^2$ (more precisely, the normalized frequency $\sqrt{\lambda_m^\Acal}/(2\pi)<1$). Thus \eqref{display:kc_settingofgamma} 
reduces to a more practical formula, where $\gamma$ is independent of the grid division $N$:
\begin{align}\label{display:kc_settingofgamma1}
\gamma=\begin{cases}
4\pi^2,\quad &\wnk=\vec0\text{ or }\normtwo{\wnk}>1,\\
4\pi^2\normtwo{\wnk}^{-2},\quad &\normtwo{\wnk}\in(0,1).
\end{cases}
\end{align}

We also point out that, for a general coordinate change matrix $\Abf$, the corresponding results of Lemma \ref{lem:kc_ABblock} (1), Proposition \ref{prop:kc_finalresult} (1) still hold. Meanwhile, the choice of $\gamma$ in \eqref{display:kc_settingofgamma1} for $\Abf=I_3$ case is also valid for our numerical examples.
\end{rmk}

In summary, Proposition \ref{prop:kc_finalresult} provides a rigorous justification 
for our kernel compensation formulation \eqref{display:kc_formulation}. It proves that for any nonzero $\wnk$, 
the null space $\Hcal$ is completely eliminated. When $\wnk=\vec 0$, $\Hcal$ 
is reduced to a three-dimensional space corresponding to constant fields. 
Moreover, it provides a simple and effective rule \eqref{display:kc_settingofgamma1} 
for choosing the penalty coefficient $\gamma$.
\begin{rmk}
A notable advantage of our finite-difference-based formulation lies in the modest penalty 
required, which leads to a relatively well-conditioned system. Our analysis 
shows $\gamma=O(1)$ is sufficient. In contrast, kernel compensation within a 
finite element framework necessitates a much larger penalty, typically $O(h^{-3})$, 
to counteract the scaling of the mass matrix. As noted in \cite{lu2022parallelpc}, 
where the penalty coefficient is chosen to be $\gamma=\frac{2}{h^3}$ in the experiment, 
spurious eigenvalues are still likely to occur. 
\end{rmk}

\subsection{Matrix-free Operations and Preconditioning}\label{sse:matfree}
In this subsection, we aim for matrix-free operation $\vec H_h\mapsto (\Acal\Mcal_\varepsilon\Acal^\conjT+\gamma\Bcal^\conjT\Bcal)\vec H_h$ and the approach 
to preconditioning the system matrix $\Acal\Mcal_\varepsilon\Acal^\conjT+\gamma\Bcal^\conjT\Bcal$ 
effectively using DFT. 
We define the $N^3\times N^3$ 3D DFT matrix $\Fcal$ as the tensor product of three 1D DFT matrices:
\begin{align}
\Fcal=\Fbf\otimes\Fbf\otimes \Fbf.
\end{align}
Applying the mixed product property \eqref{display:kc_mixedproductprop} 
and Lemma \ref{lem:kc:diagofcirc}, we establish the $\Fcal$-diagonalization of matrices $\Dcal_i$ for $i=1,2,3$.

\begin{prop}\label{prop:kc:matfreeblocks}
Assuming $\Dbf_i=\Fbf\Lambda_i\Fbf^\conjT$ for $i=0,1$, where diagonal matrices $\Lambda_i$ are derived from \eqref{display:kc_diagofcirc}, the matrices $\Dcal_i$ admit the factorization
\begin{align}\label{display:matrixfree:Dblockfac}
\Dcal_i=\Fcal\Kcal_i\Fcal^\conjT,\quad i=1,2,3,\quad\text{ where }
\left\{
\begin{aligned}
&\Kcal_1=I_{N^2}\otimes(\Lambda_1+\imagunit\wnkk_1\Lambda_0),\\
&\Kcal_2=I_N\otimes(\Lambda_1+\imagunit\wnkk_2\Lambda_0)\otimes I_N,\\
&\Kcal_3=(\Lambda_1+\imagunit\wnkk_3\Lambda_0)\otimes I_{N^2}.
\end{aligned}
\right.
\end{align}
If we further define $\Fcal_3:=\diag(\Fcal,\Fcal,\Fcal)$, the matrices $\Acal,\Bcal^\conjT\Bcal$ are factorized as:
\begin{align}\label{display:matrixfree:ABfac}
\begin{aligned}
&\Acal=\Fcal_3\Kcal_{\Acal}\Fcal_3^\conjT,\quad \Bcal^\conjT\Bcal=\Fcal_3\Kcal_{\Bcal}\Fcal_3^\conjT,\quad\text{where} \\
&\Kcal_{\Acal}=\begin{pmatrix}
\ & -\Kcal_3 &\Kcal_2\\ \Kcal_3 & \ &-\Kcal_1\\ -\Kcal_2 & \Kcal_1 & \
\end{pmatrix},\quad \Kcal_\Bcal=\begin{pmatrix}
\Kcal_1^\conjT\\\Kcal_2^\conjT\\ \Kcal_3^\conjT
\end{pmatrix}\begin{pmatrix}
\Kcal_1 & \Kcal_2 & \Kcal_3
\end{pmatrix}.
\end{aligned}
\end{align}
\end{prop}
\begin{proof}
The result follows directly from previous propositions.
\end{proof}

Since $ \Acal \Mcal_\varepsilon\Acal^\conjT+\gamma\Bcal^\conjT\Bcal = \Fcal_3( \Kcal_\Acal\Fcal^\conjT_3\Mcal_\varepsilon\Fcal_3\Kcal_\Acal^\conjT+\gamma\Kcal_\Bcal)\Fcal_3^\conjT$, 
by eliminating $\Fcal_3$ on both sides the kernel compensation formulation \eqref{display:kc_formulation} 
is simplified into the equivalent system:
\begin{align}\label{display:matrixfree_kcform}
\left(\Kcal_{\Acal}\Fcal^\conjT_3\Mcal_\varepsilon\Fcal_3\Kcal_\Acal^\conjT+\gamma \Kcal_\Bcal\right)\vec x 
= \omega_h^2 \vec x,\quad \vec x=\Fcal^\conjT_3\vec H_h.
\end{align}
To construct an effective preconditioner, we consider an approximation by neglecting the middle term $\Fcal^\conjT_3\Mcal_\varepsilon\Fcal_3$. This yields the following matrix:
\begin{align}\label{display:matrixfree_diagP}
\Kcal_\Pcal:=\Kcal_\Acal\Kcal_\Acal^\conjT+\gamma\Kcal_\Bcal.
\end{align}
Therefore, once we select an iterative algebraic eigensolver, 
a complete formulation of operations during each iteration is
\begin{align}\label{display:matrixfree_finalform}
\left\{
\begin{aligned}
&\text{Original system: }\quad\left(\Kcal_{\Acal}\Fcal^\conjT_3\Mcal_\varepsilon\Fcal_3\Kcal_\Acal^\conjT
+\gamma \Kcal_\Bcal\right)\vec x = \omega_h^2 \vec x,\\
&\text{Preconditioning: }\quad\vec x\mapsto\Kcal_\Pcal^{-1}\vec x,\quad\Kcal_\Pcal=\Kcal_\Acal\Kcal_\Acal^\conjT+\gamma\Kcal_\Bcal.
\end{aligned}
\right.
\end{align}
Crucially, the matrices $\Kcal_\Acal$ and $\Kcal_\Bcal$ are assembled using the diagonal matrices defined in \eqref{display:matrixfree:Dblockfac} and \eqref{display:matrixfree:ABfac}. 
Since $\mathcal{K}_\mathcal{P}$ is a $3\times 3$ block matrix where each block is diagonal, 
its inverse $\mathcal{K}_\mathcal{P}^{-1}$ can be trivially computed. Meanwhile, explicit 
construction of sparse matrices is unnecessary since manipulations of these $\Kcal$-matrices are reduced to piecewise vector multiplication, which significantly enhances GPU acceleration efficiency.

\begin{prop}
The condition number of the preconditioned system \eqref{display:matrixfree_finalform} 
is bounded by the condition number of the permittivity matrix $\Mcal_\varepsilon$.
\end{prop}
\begin{proof}
We consider the Rayleigh quotient $\rho(\cdot)$ ($\normtwo{\cdot}$ is vector 2-norm):
\begin{align}
\rho(\vec x):=\frac{\vec x^*\Kcal_\Acal\hat\Mcal\Kcal_\Acal^\conjT\vec x
+\gamma\normtwo{\Kcal_\Bcal\vec x}^2}{\normtwo{\Kcal_\Acal^\conjT\vec x}^2
+\gamma\normtwo{\Kcal_\Bcal\vec x}^2},\qquad \hat\Mcal=\Fcal^\conjT_3\Mcal_\varepsilon\Fcal_3,
\end{align}
where $\vec x\in\Cbb^{3N^3}$ and superscript ``$*$'' refers to vector  conjugate transpose. 
Since $\Fcal_3$ is unitary, $\hat\Mcal$ has exactly the same spectrum as $\Mcal_\varepsilon$. 
The inequality
$$
\lambda_{\min}(\Mcal_\varepsilon)\normtwo{\vec y}^2\leq\vec y^*\Mcal_\varepsilon\vec y\leq \lambda_{\max}(\Mcal_\varepsilon)\normtwo{\vec y}^2
$$
holds for any $\vec y\in\Cbb^{3N^3}$, which implies the range of the Rayleigh quotient 
is bounded by the minimum and maximum eigenvalues of $\Mcal_\varepsilon$.
\end{proof}

\subsection{Matrix-free High-order Discretization}\label{sse:kc:highord}
In the final part of this section, we investigate the design of high-order discretizations. Due to the discontinuity of the dielectric permittivity $\varepsilon^{-1}$ at material interfaces, the magnetic field $\vec H$ generally does not possess high regularity. From a theoretical perspective, neither a high-order finite element scheme using adaptive mesh nor a high-order finite difference scheme is able to reach the desired convergence order. Nevertheless, a high-order discretization can still bring unexpected improvements in accuracy under certain cases.

To design high-order discretizations, we simply replace the first rows of $\Dbf_0$ and $\Dbf_1$ with the corresponding 1D symmetric finite difference stencils. While \eqref{display:yee:D0D1} corresponds to a second order stencil, for a general $2m$-th order stencil, we have:
\begin{align}\label{display:kc:highord:stencil}
\begin{aligned}
\left(\frac{d\phi}{dx}\right)_{j-\frac{1}{2}}&= \frac{1}{h}\sum\limits_{s=1}^mc_s^{(1)}\left(\phi_{j+s-1}-\phi_{j-s}\right)+O(h^{2m}),\\
\phi_{j-\frac{1}{2}}&=\sum\limits_{s=1}^mc_s^{(0)}\left(\phi_{j+s-1}+\phi_{j-s}\right)+O(h^{2m}).
\end{aligned}
\end{align}
By eliminating lower-order terms via Taylor expansion, the coefficients $c_s^{(1)}$ and $c_s^{(0)}$ can be determined by solving the following linear system:
\begin{align}\label{display:kc:high:coef}
\begin{pmatrix}
\left(\frac{2t-1}{2}\right)^{2s+k-2}
\end{pmatrix}_{1\leq s,t\leq m}\begin{pmatrix}
c_s^{(k)}
\end{pmatrix}_{1\leq s\leq m}=\begin{pmatrix}
\frac{1}{2} & 0 & \cdots & 0
\end{pmatrix}^\top.
\end{align}
The coefficients for 1D finite difference stencils of orders 2, 4, 6, 8, and 10, derived from \eqref{display:kc:high:coef}, are summarized in Table \ref{tab:kc:fd1d}.

\begin{table}[htb]
    \centering
    \caption{Coefficients of 1D symmetric finite difference stencil}
    \label{tab:kc:fd1d}
    \begin{tabular}{ccc}
        \toprule
        $m$ & $c_1^{(1)},\cdots,c_m^{(1)}$ & $c_1^{(0)},\cdots,c_m^{(0)}$ \\
        \midrule
        1 & $1$ & $\displaystyle\frac{1}{2}$ \\
        \midrule
        2 & $\frac{9}{8},-\frac{1}{24}$ & $\frac{9}{16},-\frac{1}{16}$ \\
        \midrule
        3 & $\frac{25}{64},-\frac{25}{384},\frac{3}{640}$ & $\frac{75}{128},-\frac{25}{256},\frac{3}{256}$ \\
        \midrule
        4 & $\frac{1225}{1024},-\frac{245}{3072},\frac{49}{5120},-\frac{5}{7168}$ &
          $\frac{1225}{2048},-\frac{245}{2048},\frac{49}{2048},-\frac{5}{2048}$ \\
        \midrule
        5 & $\frac{19845}{16384},-\frac{735}{8192},\frac{567}{40960}, -\frac{405}{229376}, \frac{35}{294912}$ & $\frac{19845}{32768},-\frac{2205}{16384}, \frac{567}{16384}, -\frac{405}{65536}, \frac{35}{65536}$ \\
        \bottomrule
    \end{tabular}
\end{table}
Previous conclusions can be naturally extended to high-order cases: by using the results of Lemma \ref{lem:kc:diagofcirc} and Proposition \ref{prop:kc:matfreeblocks}, high-order discrete operators can be constructed within a matrix-free framework. Moreover, such an improvement in accuracy is achieved without causing additional time or memory expenses: In numerical implementations, we first solve coefficients $c_s^{(0)}$ and $c_s^{(1)}$ based on the desired order $2m$, then compute the diagonal matrices $\Lambda_0$, $\Lambda_1$ by Lemma \ref{lem:kc:diagofcirc} and construct diagonal blocks $\Kcal_i$ ($i=1,2,3$) by Proposition \ref{prop:kc:matfreeblocks}. The diagonal blocks $\Kcal_i$ are stored in a complex array of length $3N^3$, diagonal elements and upper diagonal elements of $\Kcal_\Bcal$, $\Kcal_\Pcal^{-1}$ are stored by two complex arrays and two real (float) arrays (of length $3N^3$). The total memory usage of storing $\Kcal_\Acal,\Kcal_\Bcal$ and $\Kcal_\Pcal^{-1}$ sums up to $192N^3$ bytes with double precision.

\begin{rmk}
    The result for null space $\Hcal$ of the kernel compensation and the penalty coefficient $\gamma$, as is discussed in Proposition \ref{prop:kc_finalresult}, holds for high-order discretization. The proof lies in analyzing the spectrum of $\Bcal\Bcal^\conjT$ with high-order stencil, which involves computing the eigenvalues of $\Kbf_i^\conjT\Kbf_i$ by Lemma \ref{lem:kc:diagofcirc}. The derivation follows the same process as the second-order case. 
\end{rmk}

\begin{rmk}
We give a brief summary of Section \ref{se:mfd}: Yee's scheme approximates the model problems \eqref{display:intro_magneticshifted} and \eqref{display:ct_magneticCT}, arriving at a preliminary result \eqref{display:ct_preliminaryresult}. 
Kernel compensation formulation \eqref{display:kc_formulation} 
eliminates the null space. Finally \eqref{display:matrixfree_finalform} provides a preconditioned 
formulation that allows matrix-free operations for implementing $\Acal,\Bcal^\conjT\Bcal$ and preconditioning. 
Formulation \eqref{display:matrixfree_finalform} is adopted for numerical
experiments in Section \ref{se:num}.
\end{rmk}

\section{Treatment of Permittivity \texorpdfstring{$\varepsilon$}{epsilon}}\label{se:permittivity}
To ensure the correctness and robustness of kernel compensation formulation \eqref{display:kc_formulation}, 
$\Mcal_\varepsilon$ is strictly required to be HPD. This section details 
how we construct the map $\vec E_h\mapsto \Mcal_\varepsilon\vec E_h$. 
Subsection \ref{sse:diagonal} discusses the simplest case for a diagonal $\varepsilon_1$. When $\varepsilon_1$ has nonzero off-diagonal entries, 
we encounter the critical issue of DoF misalignment. Subsection \ref{sse:mismatch} explains the numerical difficulty and presents two explicit solutions. Rigorous proofs of HPD properties are provided in Subsection \ref{sse:proof}. We specify the following notations to avoid ambiguity of subscripts:
\begin{align}\label{display:permittivity_notation}
\begin{aligned}
&(\Abf)_{ij}\text{: the }(i,j)\text{ entry of matrix }\Abf;\ \ 
(\Dbf)_i \text{: the }i\text{-th entry of diagonal matrix }\Dbf;\\ 
&\Mcal_{ij}\text{: the }(i,j)\text{ matrix block of a block matrix};\\
&\Mbf_k\text{: the }k\text{-th diagonal block of a block diagonal matrix}.
\end{aligned}
\end{align}
The permittivity is denoted by $\varepsilon_1=(\varepsilon_{ij})_{1\leq i,j\leq 3}$, where $\varepsilon_{ij}$'s denote the entries of $\varepsilon_1$. We also introduce edge and volume indicator matrices $\Ical_\Ecal\in\Rbb^{3N^3\times3N^3}$, $\Ical_\Vcal\in\Rbb^{N^3\times N^3}$ to distinguish the DoFs located within the material subdomain $\Omega_1$:
\begin{align}\label{display:permittivity_indicatormatrix}
\begin{aligned}
&(\Ical_{\Ecal})_{\beta}=\begin{cases}1,\quad\beta\text{-th edge DoF is located in }\Omega_1,\\
0,\quad\text{otherwise},\end{cases}\quad \Ical_\Ecal=\text{diag}(\Ical_1,\Ical_2,\Ical_3),\\&(\Ical_{\Vcal})_{\beta}=\begin{cases}1,\quad\beta\text{-th volume DoF is located in }\Omega_1,\\
0,\quad\text{otherwise}.\end{cases}
\end{aligned}
\end{align}
Matrix $\Ical_\Ecal$ can be naturally decomposed into three $N^3\times N^3$ diagonal blocks, 
where each $\Ical_i$ corresponds to the indicator matrix for edge-DoFs aligned along the $i$-th axis. 

\subsection{Diagonal Cases}\label{sse:diagonal}
The construction is straightforward when the material permittivity $\varepsilon_1$ is a diagonal matrix, i.e. $\varepsilon_1=\diag(\varepsilon_{11},\varepsilon_{22},\varepsilon_{33}),\ \varepsilon_{ii}>0$.  
In this case the dielectric matrix $\Mcal_{\varepsilon}$ is also diagonal:
\begin{align}\label{display:permittivity_diagonal_case}
\begin{aligned}
&\Mcal_{\varepsilon}=\text{diag}(\Mcal_{11},\Mcal_{22},\Mcal_{33}),\quad \Mcal_{ii}
=(\varepsilon_{ii}-1)\Ical_i+I_{N^3},
\end{aligned}
\end{align}
where $\Ical_i$ is defined in \eqref{display:permittivity_indicatormatrix}. 
We adopt a direct pointwise evaluation of $(\Mcal_{ii})_\beta$. While other approaches, 
such as that in \cite{Lyu2021fame}, 
employ a local average of surrounding values, our simpler treatment is 
sufficient to maintain the accuracy of the scheme.

\subsection{Misalignment of DoFs}\label{sse:mismatch}
A challenge arises when $\varepsilon_1$ contains nonzero off-diagonal entries 
(e.g., $\varepsilon_{12} = \varepsilon_{21}^* \neq 0$), which couple different electric field components. 
The staggered nature of the grid leads to a misalignment of DoFs: The evaluation of 
the term for the $x$-component of the electric field requires the quantity
$\varepsilon_{12} E_2$, while $E_2$ is approximated by $y$-edge DoFs, a direct 
evaluation of $E_2$ on any $x$-aligned edge is not available. This misalignment problem 
occurs for all off-diagonal coupling terms. We propose a trivial solution and 
a more physically accurate interpolation-based method to address this numerical difficulty.

\subsubsection{A Trivial Solution}\label{sssec:trivial}
The simplest approach is to completely neglect the spatial mismatch on the off-diagonal blocks and consider $\Ical_1,\Ical_2,\Ical_3$ as equivalent with $\Ical_\Vcal$. In this case we define
\begin{align}\label{display:trivial_trivial}
\mattrivial:=\begin{pmatrix}\Mcal_{11} & \varepsilon_{12}\Ical_\Vcal&\varepsilon_{13}\Ical_\Vcal\\ \varepsilon_{12}^*\Ical_\Vcal & \Mcal_{22}& \varepsilon_{23}\Ical_\Vcal\\ \varepsilon_{13}^*\Ical_\Vcal&\varepsilon_{23}^*\Ical_\Vcal&\Mcal_{33}\end{pmatrix}.
\end{align}
Although geometrically inaccurate, this approach produces a simply structured matrix, 
allowing matrix-free implementations of $\vec E_h\mapsto\mattrivial\vec E_h$.

\subsubsection{Cross-DoF Transfer via Interpolation}\label{sssec:crossdof}
A more physically faithful approach is to bridge the misalignment by interpolation. 
We can approximate the value of a field component at a desired location 
by averaging the values from its nearest neighbors on the staggered grid. 
We adopt the notations in \eqref{display:yee_dofmap}: for any $\vec E_h$ consisting of edge-DoFs, 
there holds:
\begin{align}\label{display:crossdof_xytemplate}
E^1_{i-\frac{1}{2},j,k}=\frac{1}{4}\left(E^2_{i-1,j-\frac{1}{2},k}+E^2_{i-1,j-\frac{1}{2},k}
+E^2_{i,j-\frac{1}{2},k}+E^2_{i,j+\frac{1}{2},k}\right)+O(h^2),
\end{align}
where $O(h^2)$ is the accuracy guaranteed by the smoothness of the original field $\vec E$. 
Due to periodicity, \eqref{display:crossdof_xytemplate} has a matrix representation using $\Dbf_0$ 
defined in \eqref{display:yee:D0D1}:
\begin{align}\label{display:crossdof_xymat}
\vec E_h^1=(\Dbf_0\otimes\Dbf_0^\top \otimes I_N)\vec E_h^2+O(h^2).
\end{align}
Likewise, we obtain a set of cross-DoF transfer matrices $\{\Tcal_{ij}\}_{1\leq i,j\leq 3}$ as follows:
\begin{align}\label{display:crossdof_Tmat}
\begin{aligned}
&\vec E^i_h=\Tcal_{ij}\vec E^j_h+O(h^2),\quad \vec E_h^j=\Tcal_{ij}^\top \vec E_h^i+O(h^2),\quad1\leq i<j\leq3,\\
&\Tcal_{12}=\Dbf_0\otimes\Dbf_0^\top \otimes I_N,\quad \Tcal_{13}=\Dbf_0\otimes I_N\otimes\Dbf_0^\top ,\quad \Tcal_{23}=I_N\otimes\Dbf_0\otimes\Dbf_0^\top.
\end{aligned}
\end{align}

With these transfer operators, one might think it is sufficient to simply apply the edge indicator matrices $\Ical_1,\Ical_2,\Ical_3$ afterwards. However, a naive construction of the off-diagonal block, $\varepsilon_{12}\Ical_1\Tcal_{12}$ for example, would lead to a non-Hermitian $\Mcal_\varepsilon$, as $(\varepsilon_{12}\Ical_1\Tcal_{12})^\conjT=\varepsilon_{12}^*\Tcal_{12}^\top \Ical_1$ 
does not equal the corresponding (2,1) block $\varepsilon_{12}^*\Ical_2\Tcal_{12}^\top $. 
To ensure Hermiticity, we construct symmetrized cross-DoF transfer operators and formulate the matrix as:
\begin{align}\label{display:crossdof_herm}
\matcrossdof
:=\begin{pmatrix}
\Mcal_{11} & \varepsilon_{12}\Scal_{12}&\varepsilon_{13}\Scal_{13}\\ \varepsilon_{12}^*\Scal_{12}^\top  
& \Mcal_{22}& \varepsilon_{23}\Scal_{23}\\ \varepsilon_{13}^*\Scal_{13}^\top &\varepsilon_{23}^*\Scal_{23}^\top 
&\Mcal_{33}\end{pmatrix},\quad \Scal_{ij}=\frac{\Ical_i\Tcal_{ij}+\Tcal_{ij}\Ical_j}{2},\quad 1\leq i<j\leq3.
\end{align}
In practice, $\matcrossdof$ is explicitly assembled as a sparse matrix due to its relatively 
complicated structure. Although this solution consumes slightly more storage, the efficiency 
is barely affected, as is shown in the experiments.

\subsection{Proof of HPD Property}\label{sse:proof}
We provide analysis for the HPD properties of the discrete permittivity matrix 
$\mattrivial$ and $\matcrossdof$, which are expected to be HPD whenever the $3\times3$ matrix $\varepsilon_1$ is HPD. Unfortunately, 
the HPD property may not be preserved without the following assumptions:
\begin{assump}\label{display:proof_assump_eigen}
Eigenvalues of $\varepsilon_1$ lie in $(0,1]$.
\end{assump}

\begin{assump}\label{display:proof_assump_sdd}
Matrix $\varepsilon_1$ is strictly diagonally dominant (SDD): 
$\varepsilon_{ii}>\sum\limits_{j\neq i}|\varepsilon_{ij}|$, $i=1,2,3$.
\end{assump}

\begin{assump}\label{display:proof_assump_nonzero}
At least one off-diagonal entry of $\varepsilon_1$ is zero.
\end{assump}

The above assumptions are reasonable in physics. The $3\times3$ matrix
$\varepsilon_1$ represents the inverse of the normalized permittivity tensor $\varepsilon$ 
within the material domain $\Omega_1$. For any physical material, $\varepsilon$ must have 
eigenvalues no less than 1, with 1 corresponding to the vacuum. Consequently, the spectrum 
of $\varepsilon_1$ must lie in $(0,1]$, as is illustrated by Assumption \ref{display:proof_assump_eigen}. 
In anisotropic media, off-diagonal entries in $\varepsilon_1$ often arise 
from a rotation of the material's principal axes. A common case involves 
a material with anisotropy rotated only in the $xy$-plane, which leads 
to $\varepsilon_{13}=\varepsilon_{23}=0$. Assumption \ref{display:proof_assump_nonzero} 
is thus satisfied in this case. Assumption \ref{display:proof_assump_sdd} is also reasonable 
since many anisotropic materials exhibit strong diagonal dominance.

In Subsubsection \ref{sssec:prooftrivial} we prove that $\mattrivial$ 
is HPD under Assumption \ref{display:proof_assump_eigen} alone; Subsubsection \ref{sssec:proofcrossdof} 
proves that $\matcrossdof$ is HPD if Assumption \ref{display:proof_assump_eigen} 
is combined with either Assumption \ref{display:proof_assump_sdd} or \ref{display:proof_assump_nonzero}.

\subsubsection{Positive Definiteness for Matrix \texorpdfstring{$\mattrivial$}{M-Trivial}}\label{sssec:prooftrivial}
In the proofs that follow, we adhere to the notations displayed in \eqref{display:permittivity_notation}.
\begin{lem}\label{lem:proof_permmat}
For any matrix $A\in \Cbb^{N\times N}$ and permutation $\sigma$ over $\{1,\cdots,N\}$, 
if matrix $\hat A$ satisfies $(\hat A)_{ij}=(A)_{\sigma(i),\sigma(j)}$, 
then $\hat A=PAP^\top $ where permutation matrix $P$ is such 
that $(P)_{ij}=\delta_{\sigma(i),j}$ ($\delta$ is the Kronecker delta).
\end{lem}
\begin{proof}
We have $(PA)_{ij}=\sum\limits_{k=1}^N(P)_{ik}(A)_{kj}=(A)_{\sigma(i),j}$ 
and $(\hat AP)_{ij}=\sum\limits_{k=1}^N(\hat A)_{ik}(P)_{kj}=(\hat A)_{i,\sigma^{-1}(j)}$, 
thus $(PA)_{ij}=(\hat AP)_{ij}$ holds for any $1\leq i,j\leq N$. 
\end{proof}
\begin{lem}\label{lem:proof_blockdiageigen}
Let $\Mcal$ be an $mN\times mN$ Hermitian block matrix partitioned into $m\times m$ 
blocks $\Mcal=(\Mbf_{ij})_{1\leq i,j\leq m}$, each block $\Mbf_{ij}$ is $N\times N$ 
diagonal and $\Mbf_{ji}=\Mbf_{ij}^\dagger$. Then there holds:
\begin{align}
\lambda(\Mcal)=\bigcup\limits_{k=1}^N\lambda(\hat{\Mbf}_k),\quad\text{where}\ \hat{\Mbf}_k\in\Cbb^{m\times m},\ (\hat{\Mbf}_k)_{ij}=(\Mbf_{ij})_k.
\end{align}
\end{lem}
\begin{proof}
We introduce matrix $\hat\Mcal$ such that $(\hat\Mcal)_{ij}=(\Mcal)_{\sigma(i),\sigma(j)}$, 
where $\sigma$ is a permutation over $\{1,\cdots,mN\}$ defined by:
\begin{align}
\sigma(i)=(i-(k-1)m-1)N+k,\quad (k-1)m+1\leq i\leq km,\quad k=1,\cdots,N.
\end{align}
According to the diagonal condition, $(\Mcal)_{ij}\neq0$ if and only if $i\equiv j\mod N$,
thus $(\hat\Mcal)_{ij}\neq 0$ if and only if $\sigma(i)\equiv\sigma(j)\mod N$, 
which means $i,j$ lay in the same interval $[(k-1)m+1,km]$, which is equivalent to saying
\begin{align}
\hat\Mcal=\diag(\hat{\Mbf}_1,\cdots,\hat{\Mbf}_N),\quad \hat{\Mbf}_k\in\Cbb^{m\times m}.
\end{align}
Meanwhile, $(\hat{\Mbf}_k)_{ij}=(\hat{\Mcal})_{(k-1)m+i,(k-1)m+j}=(\Mcal)_{(i-1)N+k,(j-1)N+k}=(\Mbf_{ij})_k$, also $(\hat{\Mbf}_k)_{ji}=(\Mbf_{ji})_k=(\Mbf_{ij}^\conjT)_k=(\hat{\Mbf}_k)_{ij}^\conjT$, 
thus $\hat{\Mbf}_k$ is also Hermitian. Finally, according to Lemma \ref{lem:proof_permmat}, 
$\Mcal$ shares exactly same eigenvalues with $\hat\Mcal$, thus the desired result holds.
\end{proof}

\begin{prop}\label{prop:proof_trivialHPD} 
The matrix $\mattrivial$ defined by \eqref{display:trivial_trivial} 
in Subsubsection \ref{sssec:trivial} satisfies
\begin{align}\label{display:proof_lowboundsmeig}
\lambda_{\min}(\mattrivial)\geq\lambda_{\min}(\hat\varepsilon_1),
\end{align}
where 
\begin{align}\label{display:proof_HPDcond1}
\hat\varepsilon_1:=\begin{pmatrix}
\hat\varepsilon_{11} & \varepsilon_{12} & \varepsilon_{13} \\
\varepsilon_{12}^* &\hat\varepsilon_{22} &\varepsilon_{23}\\
\varepsilon_{13}^* & \varepsilon_{23}^*&\hat\varepsilon_{33}
\end{pmatrix},\quad \hat\varepsilon_{ii}=\min\{\varepsilon_{ii},1\}.
\end{align}
\end{prop}
\begin{proof}
Due to Lemma \ref{lem:proof_blockdiageigen}, 
we have $\lambda_{\min}(\mattrivial)=\min\limits_{1\leq k\leq N}\lambda_{\min}(\Mbf_k)$, where
\begin{align}
\Mbf_k=\begin{pmatrix}
(\varepsilon_{11}-1)(\Ical_1)_k+1 & \varepsilon_{12}(\Ical_\Vcal)_k & \varepsilon_{13}(\Ical_\Vcal)_k \\
\varepsilon_{12}^*(\Ical_\Vcal)_k& (\varepsilon_{22}-1)(\Ical_2)_k+1& \varepsilon_{23}(\Ical_\Vcal)_k\\
\varepsilon_{13}^*(\Ical_\Vcal)_k & \varepsilon_{23}^*(\Ical_\Vcal)_k & (\varepsilon_{33}-1)(\Ical_3)_k+1
\end{pmatrix}.
\end{align}
When $(\Ical_\Vcal)_k=0$, $\Mbf_k$ is diagonal with positive entries; 
when $(\Ical_\Vcal)_k=1$, $\Mbf_k\geq\hat\varepsilon_1$. Therefore
$\lambda_{\min}(\mattrivial)\geq\min\{\lambda_{\min}(\hat\varepsilon_1)$, 
$\hat\varepsilon_{11}$, $\hat\varepsilon_{22}$, $\hat\varepsilon_{33}\}$. 
Moreover, $\hat\varepsilon_{ii}=\vec e_i^\top \hat\varepsilon_1\vec e_i\geq \lambda_{\min}(\hat\varepsilon_1)$ 
due to minimality of Rayleigh quotient, 
we have $\lambda_{\min}(\mattrivial)\geq\lambda_{\min}(\hat\varepsilon_1)$. 
\end{proof}

Proposition \ref{prop:proof_trivialHPD} proves that 
Assumption \ref{display:proof_assump_eigen} is sufficient for $\mattrivial$ to be HPD.

\subsubsection{Positive Definiteness for Matrix \texorpdfstring{$\matcrossdof$}{M-CrossDoF}}\label{sssec:proofcrossdof}
We start by considering $\Scal_{ij}$'s matrix norms, which can help prove 
the HPD property under the SDD condition of Assumption \ref{display:proof_assump_sdd}.
\begin{prop}\label{prop:proof_Smatrix}
The matrices $\{\Scal_{ij}\}_{1\leq i<j\leq 3}$ defined in \eqref{display:crossdof_herm} satisfies:
\begin{align}\label{display:proof_Smatrixnrm}
\|\Scal_{ij}\|_1\leq1,\quad\|\Scal_{ij}\|_2\leq1,\quad \|\Scal_{ij}\|_\infty\leq1.
\end{align}
Here, $\|\cdot\|_1,\|\cdot\|_2,\|\cdot\|_\infty$ refer to the matrix column sum, 
spectral norm and row sum, respectively. 
\end{prop}
\begin{proof}
It suffices to prove the same properties for $\Dbf_0$. $\|\Dbf_0\|_1=\|\Dbf_0\|_{\infty}=1$ is trivial due to \eqref{display:yee_block}. By Lemma \eqref{lem:kc:diagofcirc}, it is known that $\|\Dbf_0\|_2=1$. 
Therefore Kronecker 2-norm of products of $\Dbf_0,\Dbf_0^\top $ and $I_N$ also equals 1, 
which implies $\|\Scal_{ij}\|_2\leq1$. 
Next we verify that $\|A\otimes B\|_\infty=\|A\|_\infty\|B\|_\infty$ 
for any matrix $A$, $B$. If this holds, then $\|\Scal_{ij}\|_\infty\leq1$ holds. 
Also $\|A\otimes B\|_1=\|A^\top \otimes B^\top \|_\infty=\|A^\top \|_\infty\|B^\top \|_\infty=\|A\|_1\|B\|_1$, 
which implies $\|\Scal_{ij}\|_1\leq1$.

For simplicity we assume $A=(a_{ij})$, $B=(b_{st})\in\Rbb^{N\times N}$. 
Since $(A\otimes B)_{I,J}=a_{ij}b_{st}$, where $I=(i-1)N+s,J=(j-1)N+t$, thus 
$$
\|A\otimes B\|_\infty=\max_I \sum\limits_{J}|(A\otimes B)_{I,J}|
=\max_{i,s}\sum\limits_{j,t}^N|a_{ij}b_{st}|
=\max_i\sum\limits_j|a_{ij}|\max_s\sum\limits_t|b_{st}|=\|A\|_\infty\|B\|_\infty,
$$
where $I,J$ vary from $1$ to $N^2$ 
and $i,j,s,t$ vary from $1$ to $N$. The proof is completed.
\end{proof}

\begin{prop}\label{prop:proof_sdd}
If Assumptions \ref{display:proof_assump_eigen} and \ref{display:proof_assump_sdd} hold, then $\matcrossdof$ is SDD. Hence $\matcrossdof$ is HPD according to Gerschogrin's theorem \cite[Theorem 7.2.1]{golub13}.
\end{prop}
\begin{proof}
$\matcrossdof$ being SDD is a straightforward result 
from $\|\Scal_{ij}\|_1,\|\Scal_{ij}\|_\infty\leq1$ by shrinking the diagonal 
blocks $\Mcal_{ii}$ to $\varepsilon_{ii}I_{N^3}$.
\end{proof}

Next we prove a general conclusion in verifying positive-definiteness of Hermitian 
block matrix to verify the sufficiency of Assumption \ref{display:proof_assump_nonzero}, 
as is illustrated in the following lemma.
\begin{lem}\label{lem:proof_HPDblockmatrix}
We assume a set of $m^2$ matrices $\{\Scal_{ij}\}_{1\leq i<j\leq m}$, 
where $\Scal_{ij}\in\Cbb^{N\times N}$, to satisfy: $\Scal_{ii}=I_N$ 
for $i=1,\cdots,N$ and $\Scal_{ji}=\Scal_{ij}^\conjT,\ \|\Scal_{ij}\|_2\leq1$ 
for $1\leq i<j\leq N$. We define $mN\times mN$ block matrix $\Mcal$ 
with blocks $\Mcal:=\{\varepsilon_{ij}\Scal_{ij}\}_{1\leq i,j\leq m}$, 
where $(\varepsilon_{ij})$ forms an $m\times m$ HPD matrix. 
If $\varepsilon$ is tridiagonal, then $\Mcal$ is HPD.
\iffalse
(2). When $m=3$ and off-diagonals $\varepsilon_{12}, $\varepsilon_{13}, $\varepsilon_{23}$ 
are nonzeros. The following provides a sufficient condition for $\Mcal$ to be HPD:
\begin{align}\label{display:proof_HPDcond3times3}
\arg(\varepsilon_{12})-\arg(\varepsilon_{13})+\arg(\varepsilon_{23}) 
\text{ is odd multiples of }\pi,\text{ range of arg lays in }(-\pi,\pi].
\end{align}
\fi
\end{lem}

\begin{proof}
For any $z_1,\cdots,z_m\in\Cbb$, we denote $z_i=a_i\exp(\imagunit\theta_i)$ with $a_i=|z_i|$ 
and $\theta_i=\arg z_i$. Then
\begin{align}
\begin{aligned}
0<(z_1^*,\cdots,z_m^*)\varepsilon\begin{pmatrix}
z_1\\ \vdots \\z_m
\end{pmatrix}&=\sum\limits_{i=1}^m\varepsilon_{ii}a_i^2+2\sum\limits_{i=1}^{m-1}
\text{Re}(\varepsilon_{i,i+1}z_i^*z_{i+1})\\
&=\sum\limits_{i=1}^m\varepsilon_{ii}a_i^2+2\sum\limits_{i=1}^{m-1}
|\varepsilon_{i,i+1}|a_ia_{i+1}\cos\left(\theta_{i+1}-\theta_i+\arg(\varepsilon_{i,i+1})\right).
\end{aligned}
\end{align}
By inductively setting $\theta_1=0$ and $\theta_{i+1}=\theta_i-\arg(\varepsilon_{i,i+1})+\pi$, we obtain
\begin{align}\label{display:proof_advancedcondition}
\sum\limits_{i=1}^m\varepsilon_{ii}a_i^2-2\sum\limits_{i=1}^{m-1}|\varepsilon_{i,i+1}|a_ia_{i+1}>0,
\quad \forall a_1,\cdots,a_m\geq0,\ (a_1,\cdots,a_m)\neq\vec0.
\end{align}
Then for any $\vec x_1,\cdots,\vec x_m\in\Cbb^{N}$, there holds
\begin{align}\label{display:proof_HPDblockmatrix}
\begin{aligned}
(\vec x_1^*,\cdots,\vec x_m^*)\Mcal\begin{pmatrix}
\vec x_1\\ \vdots\\ \vec x_m
\end{pmatrix}&=\sum\limits_{i=1}^m\varepsilon_{ii}\normtwo{\vec x_i}^2
+2\sum\limits_{i=1}^{m-1}\text{Re}(\varepsilon_{i,i+1}\vec x_i^*\Scal_{i,i+1}\vec x_{i+1})\\
&\geq\sum\limits_{i=1}^m\varepsilon_{ii}\normtwo{\vec x_i}^2
-2\sum\limits_{i=1}^{m-1}|\varepsilon_{i,i+1}|\normtwo{\vec x_i}\normtwo{\vec x_{i+1}},
\end{aligned}
\end{align}
where we apply the condition $\|\Scal_{ij}\|_2\leq 1$. Thus \eqref{display:proof_advancedcondition} 
implies that right-hand side of \eqref{display:proof_HPDblockmatrix} 
is nonnegative and equals zero if and only if all $\vec x_i$ vanish. 
\end{proof}

With the help of Lemma \ref{lem:proof_HPDblockmatrix}, the sufficiency of 
Assumption \ref{display:proof_assump_nonzero} is apparent.
\begin{prop}
When Assumptions \ref{display:proof_assump_eigen} and \ref{display:proof_assump_nonzero}  
hold, $\matcrossdof$ defined by \eqref{display:crossdof_herm} is HPD.
\end{prop}
\begin{proof}
If at least one off-diagonal entries of $\varepsilon_1$ vanish, 
then $\varepsilon_1$ is either tridiagonal or similar to a tridiagonal 
matrix due to Lemma \ref{lem:proof_permmat}. 
Thus by shrinking diagonal entries of $\Mcal_\varepsilon$ to $\min\{1,\varepsilon_{ii}\}=\varepsilon_{ii}$, 
applying Proposition \ref{prop:proof_Smatrix} and Lemma \ref{lem:proof_HPDblockmatrix}, we conclude that $\Mcal_\varepsilon$ is HPD.
\end{proof}

\section{Numerical experiments}\label{se:num}
All calculations are carried out based on the matrix-free kernel compensation 
formulation \eqref{display:matrixfree_finalform} with penalty coefficient $\gamma$ 
set by \eqref{display:kc_settingofgamma1}. We consider the simple cubic (SC), 
face-centered cubic (FCC) and body-centered cubic (BCC) lattices. 
The primitive lattice vectors ($\vec a_1, \vec a_2, \vec a_3$) of the 
lattice types, which can be found in \cite{joannopoulos2008book}, are specified as:
\begin{align}\label{display:exp_latticeconst}
\left\{
\begin{aligned}
\text{SC lattice:}\quad & \vec a_1 = (1, 0, 0)^\top , \quad \vec a_2 = (0, 1, 0)^\top , 
\quad \vec a_3 = (0, 0, 1)^\top,\\
\text{FCC lattice:}\quad & \vec a_1 = (0, \frac{1}{2}, \frac{1}{2})^\top,  
\quad \vec a_2 = (\frac{1}{2}, 0, \frac{1}{2})^\top , \quad \vec a_3 = (\frac{1}{2}, \frac{1}{2}, 0)^\top, \\
\text{BCC lattice:}\quad & \vec a_1 = (-\frac{1}{2}, \frac{1}{2}, \frac{1}{2})^\top, 
\quad \vec a_2 = (\frac{1}{2}, -\frac{1}{2}, \frac{1}{2})^\top, 
\quad \vec a_3 = (\frac{1}{2}, \frac{1}{2}, -\frac{1}{2})^\top.
\end{aligned}
\right.
\end{align}

The wave number vector $\wnk$ in model problem \eqref{display:intro_magneticshifted} 
is derived from symmetry points and partition points on their connecting lines in the Brillouin zone. 
These symmetry points are:
\begin{align}\label{display:exp_bzsym}
\left\{
\begin{aligned} 
\text{SC: }\quad & K(0, 0, 0), L(\pi,0,0), M(\pi,\pi, 0), N(\pi,\pi,\pi), \\
\text{FCC: }\quad & X(0, 2\pi, 0), U(\frac{\pi}{2},2\pi, \pi), L(\pi,\pi,\pi), 
\Gamma(0, 0, 0), W(\pi,2\pi, 0), K(\frac{3\pi}{2}, \frac{3\pi}{2}, 0), \\
\text{BCC: }\quad & H'(2\pi, 0, 0), \Gamma(0, 0, 0), P(\pi,\pi,\pi), N(\pi, 0, \pi), H(0, 2\pi, 0). 
\end{aligned}
\right.
\end{align}

The primitive lattice vectors \eqref{display:exp_latticeconst} 
and symmetry points \eqref{display:exp_bzsym} are given in dimensionless units, 
which corresponds to setting the lattice constant to unity. 
This choice is independent of any physical scaling and is made without 
loss of generality, as such a scaling constant does not affect the results or 
efficiency. We investigate four different lattice structures, 
the unit cells of which are visually presented in Figure \ref{fig:exp_4materials}. 
These models can be found in \cite{huang2013eigendecomppc,lu2013weyl, MEHL2017S1}.

\begin{figure} 
\centering 
\begin{subfigure}[b]{0.48\textwidth} 
\centering
\includegraphics[width=0.9\textwidth]{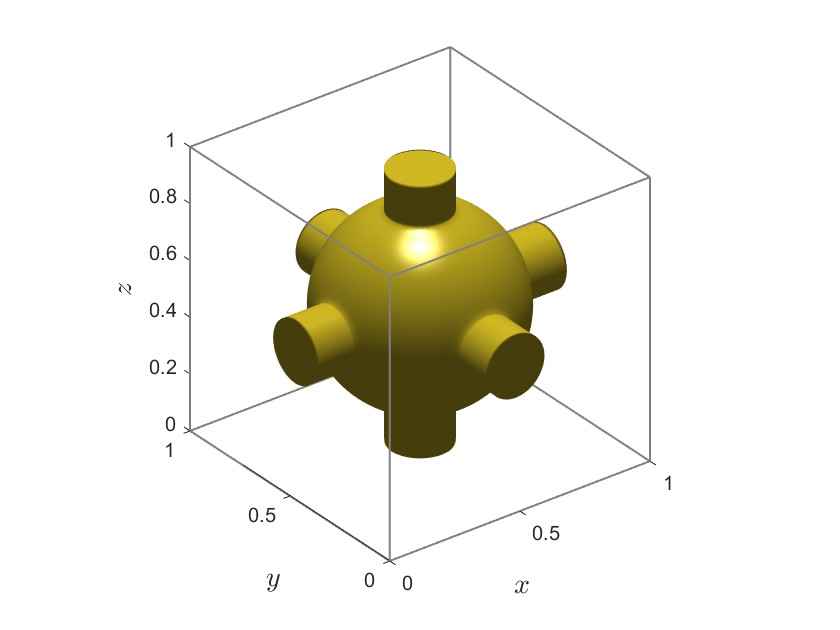} 
\caption{SC-CURV.}
\label{fig:material_sccurv}
\end{subfigure}
\hfill
\begin{subfigure}[b]{0.48\textwidth}
\centering
\includegraphics[width=0.9\textwidth]{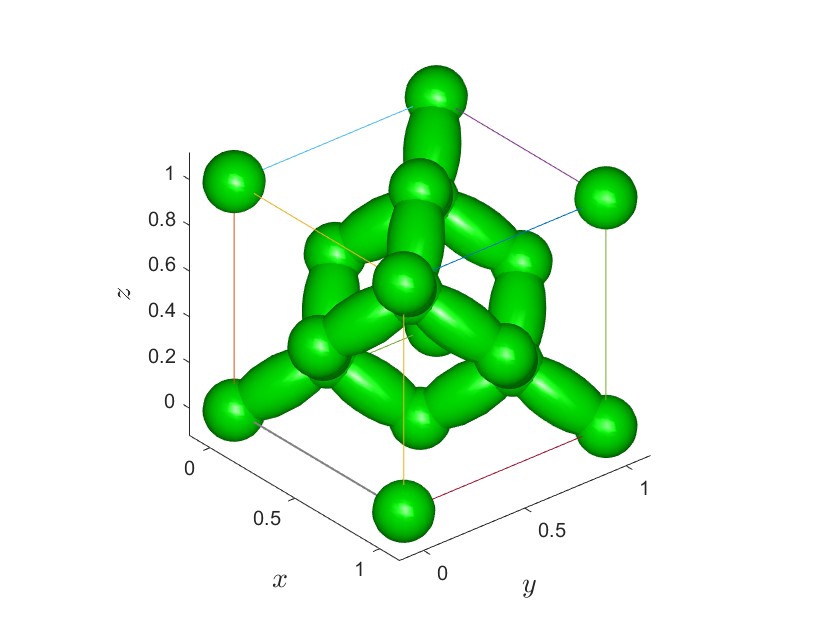}
\caption{FCC.}
\label{fig:material_fcc}
\end{subfigure}

\vspace{1em} 

\begin{subfigure}[b]{0.48\textwidth}
\centering
\includegraphics[width=0.9\textwidth]{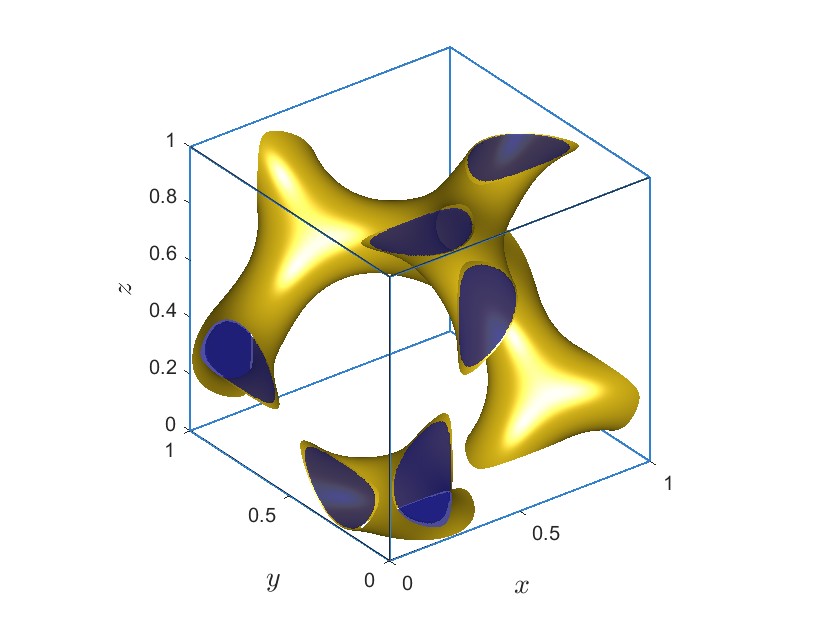}
\caption{BCC-SG.}
\label{fig:material_bccsg}
\end{subfigure}
\hfill 
\begin{subfigure}[b]{0.48\textwidth}
\centering
\includegraphics[width=0.9\textwidth]{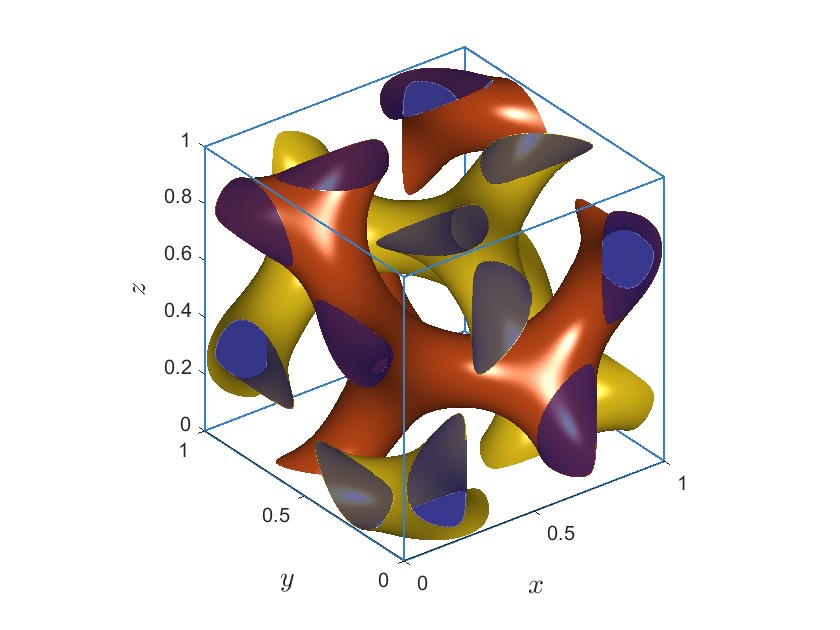}
\caption{BCC-DG.}
\label{fig:material_bccdg}
\end{subfigure}

\caption{Figures of material structures of four sample lattices in cubic domain $[0,1]^3$.}
\label{fig:exp_4materials}
\end{figure}
For simplicity, we will refer to these structures by abbreviations: 
\begin{itemize}
\item \textbf{SC-CURV (SC lattice with curved surface)}: The SC lattice shown 
in Figure \ref{fig:material_sccurv} is composed of a sphere with radius $0.345$ 
and four cylinders with base radius $0.11$.
\item \textbf{FCC}: The FCC lattice in Figure \ref{fig:material_fcc} 
adopts a diamond structure. Its material configuration is $sp^3$-like, 
consisting of spheres with radius of $0.12$ and connecting 
spheroids of minor axis $0.11$. The foci of the spheroids are located 
at the centers of the spheres.
\item \textbf{BCC-SG (BCC lattice with single gyroid surface)}: The BCC lattice 
in Figure \ref{fig:material_bccsg} has its material domain specified 
by $\left\{\vec x\in\Omega:g(\vec x)>1.1\right\}$, where
\begin{align*}
g(x,y,z)=\sin\left(2\pi x\right)\cos\left(2\pi y\right)+\sin\left(2\pi y\right)\cos\left(2\pi z\right)+\sin\left(2\pi z\right)\cos\left(2\pi x\right).
\end{align*}
\item \textbf{BCC-DG (BCC lattice with double gyroid surface)}: Similar to BCC-SG, 
the material domain of BCC lattice shown in Figure \ref{fig:material_bccdg} 
is specified by $\left\{\vec x\in\Omega:|g(\vec x)|>1.1\right\}$.
\end{itemize}

The matrix eigenproblem \eqref{display:matrixfree_finalform} is solved 
by a custom-implemented LOBPCG (Locally Optimal Block Preconditioned Conjugate Gradient) solver with soft locking \cite{knyazev2001lobpcg} in Python. 
We define the relative residual for the $j$-th eigenvalue-eigenvector 
pair $(\omega_{h,j}^2,\vec x_j)$ as
\begin{align}\label{display:exp_relativeresidual}
\text{Res}_j:=\frac{\normtwo{(\Acal\Mcal_\varepsilon\Acal^\conjT
+\gamma\Bcal^\conjT\Bcal)\vec x_j-\omega_{h,j}^2\vec x_j}}{\normtwo{\vec x_j}},
\quad \normtwo{\cdot}\text{ refers to the 2-norm}.
\end{align}
Convergence is considered to be reached when $\text{Res}_j\leq10^{-5}$ for all desired eigenpairs.

This section is outlined as: In Subsection \ref{sse:num:band}, we plot the band structures 
and analyze the differences adopting $\mattrivial$ and $\matcrossdof$. In Subsection \ref{sse:num:highord}
In Subsection \ref{sse:num:perf} the iteration count, runtime and speedup by GPU 
acceleration efficiency are recorded and analyzed. Finally we simulate extreme cases 
to demonstrate the robustness of our scheme in Subsection \ref{sse:num:custom}. 
All programs are implemented on a platform equipped with 
NVIDIA GeForce RTX 4090 D GPUs, 24 GiB of memory, 
utilizing CUDA Version 12.4 and Driver Version 550.144.03 for GPU 
acceleration. The source code has been made publicly available 
on GitHub  \url{https://github.com/Epsilon-79th/linear-eigenvalue-problems-in-photonic-crystals}.

\subsection{Band Structures of isotropic and Anisotropic Systems}\label{sse:num:band}
The expression of $\varepsilon_1$, which refers to the inverse of permittivity inside $\Omega_1$ as we define in \eqref{display:intro_permittivity}, is chosen to be 
\begin{align}\label{display:exp_permittivity}
\begin{aligned}
&\varepsilon_1=\varepsilon_{\text{lattice}}^{-1}I_3,\quad \text{for isotropic systems},\\
&\varepsilon_1 = \varepsilon_{\text{lattice}}^{-1}\begin{pmatrix}
\sqrt{1+|\beta|^2} & -|\beta|\imagunit & \ \\
|\beta|\imagunit & \sqrt{1+|\beta|^2} & \ \\
\ & \ & 1
\end{pmatrix},\quad\text{for anisotropic systems},
\end{aligned}
\end{align}
where the external magnetic field intensity $\beta=0.875$,
$\varepsilon_{\text{lattice}}$ is a constant depending on the type of 
lattice. We set $\varepsilon_{\text{lattice}}=13$ for the SC and FCC 
lattice and $\varepsilon_{\text{lattice}}=16$ for the BCC lattice. 
We compute the first 10 smallest frequencies with grid division $N=120$ 
and discrete permittivity matrix $\Mcal_\varepsilon=\matcrossdof$. 
The band structures are presented in Figures \ref{fig:band_sc_curv}--\ref{fig:band_bccdg}, 
where first 10 eigenpairs are computed for each material. 
Bandgap is an interval where no frequency exists, as is highlighted in yellow 
in the figures. The ratio of the bandgap is defined by
\begin{align}\label{display:exp_bandgapratio}
\text{gap ratio}:=\frac{\omega_{\text{up}}-\omega_{\text{low}}}{(\omega_{\text{up}}+\omega_{\text{low}})/2},
\end{align}
where $\omega_{\text{up}}$ is the minimum frequency 
above the bandgap while $\omega_{\text{low}}$ is 
the maximum frequency below the bandgap. The values of gap ratio of different 
PCs are marked in the title, 
as is shown in Figures \ref{fig:band_sc_curv}--\ref{fig:band_bccdg}. 
It can be observed that the gap ratios of SC-CURV, FCC and BCC-SG 
shrink with the presence of off-diagonal term $\varepsilon_{12}=-|\beta|\imagunit$. 
For BCC-DG, the existence of Weyl points shown in Figure \ref{fig:band_bccdg} 
coincides with the results in \cite{lu2013weyl}. 

\begin{figure}
\centering
\includegraphics[width=\linewidth]{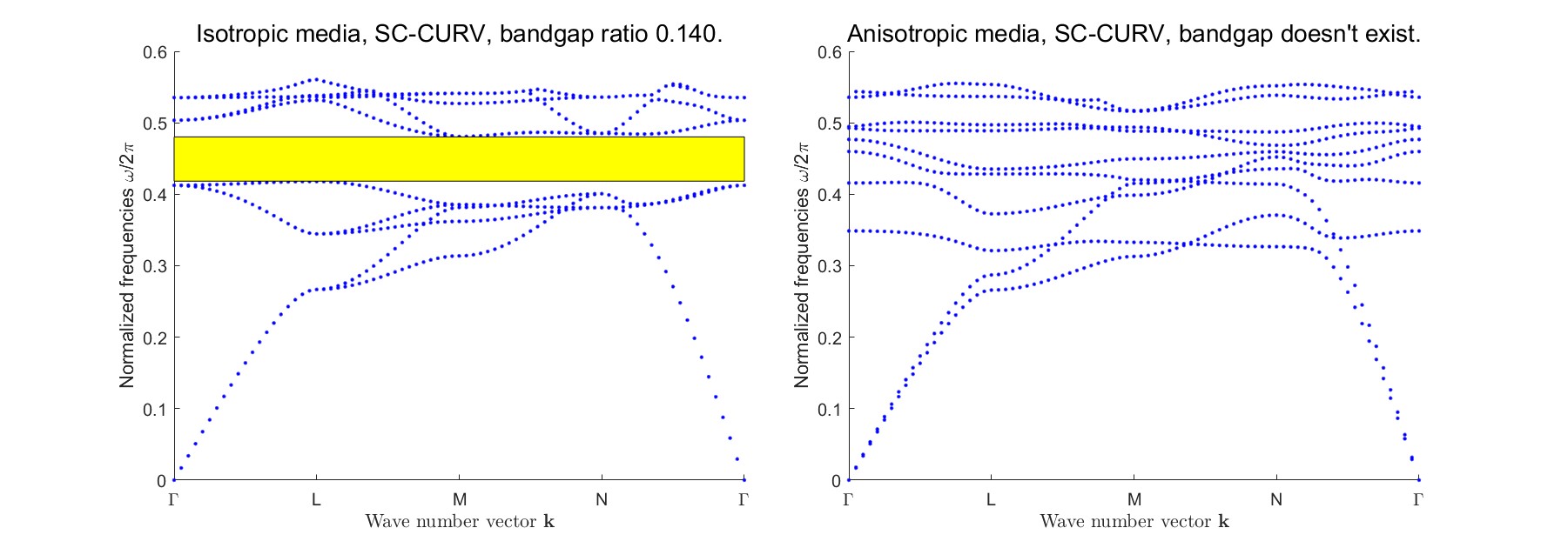}
\caption{Band structures of SC lattice with curved material interface. 
Left: isotropic system; Right: anisotropic system. Grid division $N=120$.}
\label{fig:band_sc_curv}
\end{figure}

\begin{figure}
\centering
\includegraphics[width=\linewidth]{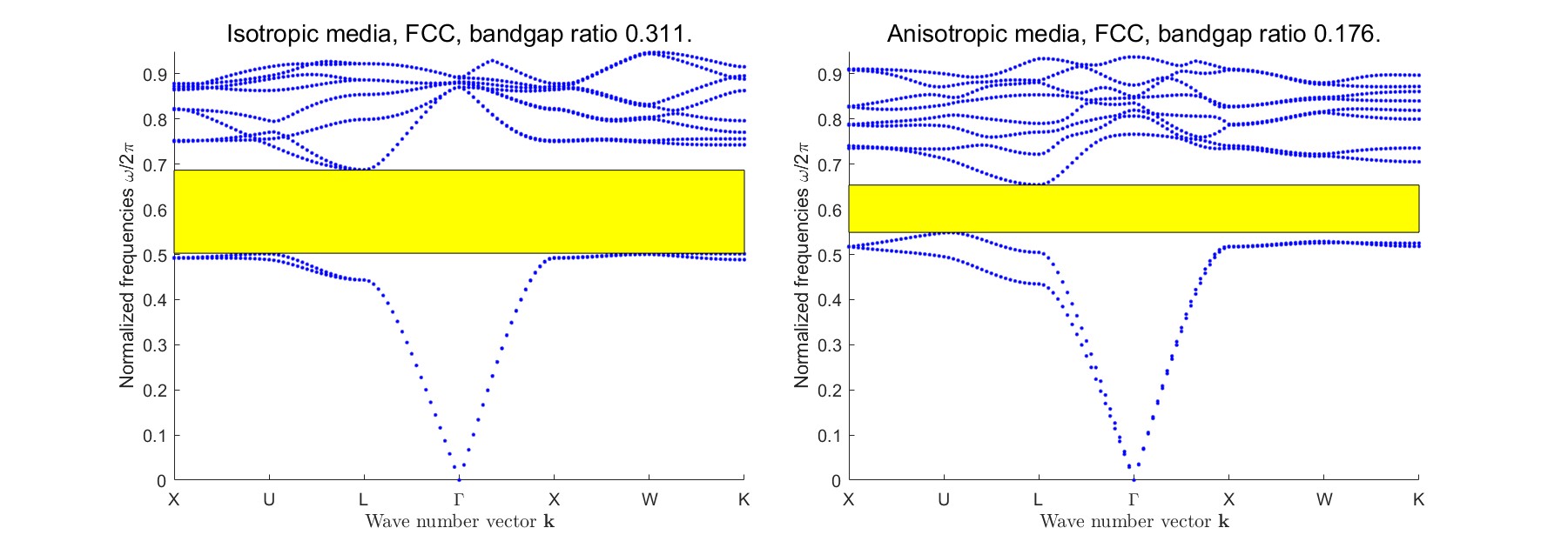}
\caption{Band structures of FCC lattice. Left: isotropic system; 
Right: anisotropic system. Grid division $N=120$.}
\label{fig:band_fcc}
\end{figure}

\begin{figure}
\centering
\includegraphics[width=\linewidth]{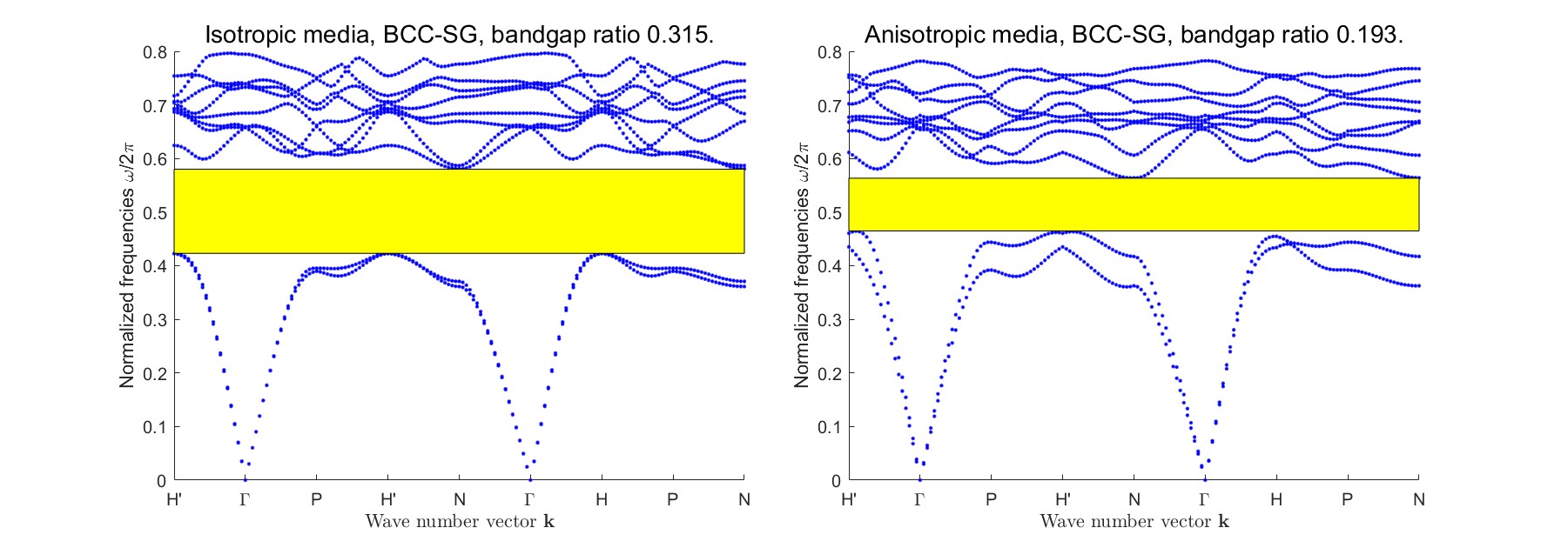}
\caption{Band structures of BCC lattice with single gyroid surfaces. 
Left: isotropic system; Right: anisotropic system. Grid division $N=120$.}
\label{fig:band_bccsg}
\end{figure}

\begin{figure}
\centering
\includegraphics[width=\linewidth]{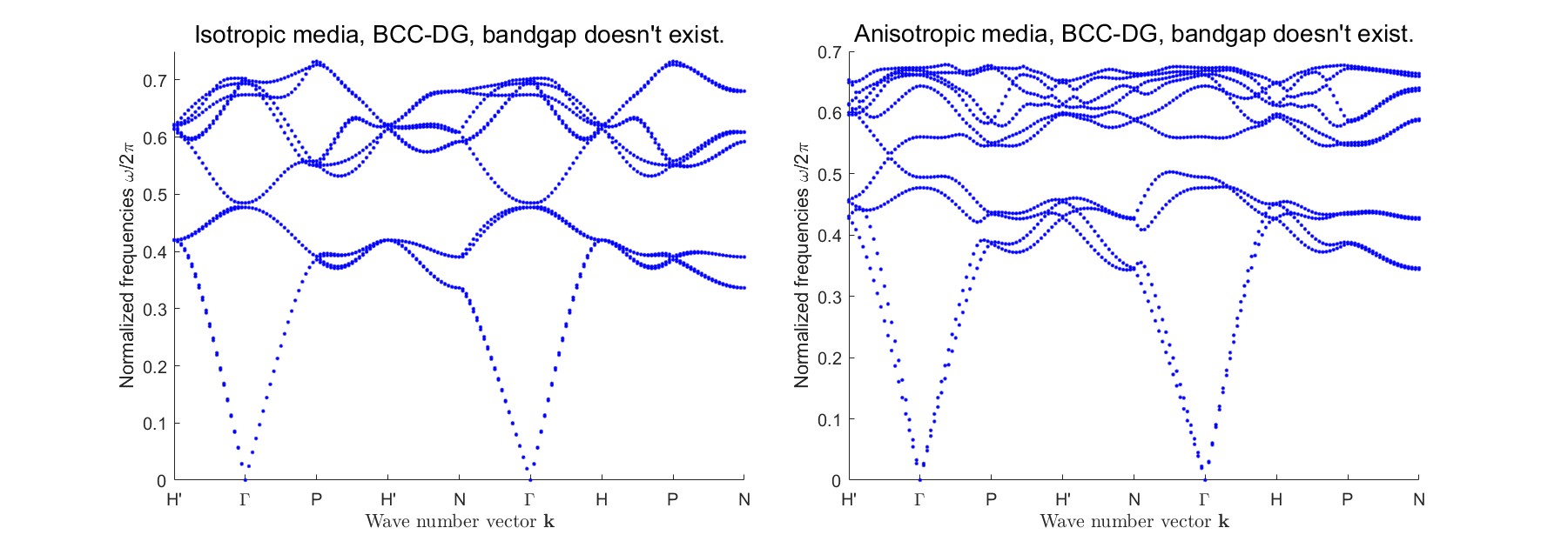}
\caption{Band structures of BCC lattice with double gyroid surfaces. 
Left: isotropic system; Right: anisotropic system. Grid division $N=120$.}
\label{fig:band_bccdg}
\end{figure}

We also compute complete band structures of the above models with 
$\Mcal_\varepsilon=\mattrivial$. It is hard to perceive 
the difference directly if both band structures are plotted on the 
same figure. Instead we present their numerical difference in 
Table \ref{tab:frequency_difference}. The maximum and mean relative 
difference $\Delta \omega_{\max}$, $\Delta \omega_{\text{mean}}$ are defined by
\begin{align}\label{display:exp_omegadiff}
\Delta \omega_{\max}=\max\limits_{\wnk,j}\left\{\frac{|\omega_{h,j}^{(1)}(\wnk)
-\omega_{h,j}^{(2)}(\wnk)|}{|\omega_{h,j}^{(2)}(\wnk)|}\right\},\quad \Delta \omega_{\text{mean}}=\mathop{\text{mean}}\limits_{\wnk,j}
\left\{\frac{|\omega_{h,j}^{(1)}(\wnk)-\omega_{h,j}^{(2)}(\wnk)|}{|\omega_{h,j}^{(2)}(\wnk)|}\right\},
\end{align}
where $\omega_{h,j}^{(1)}(\wnk),\omega_{h,j}^{(2)}(\wnk)$ denote the $j$-th 
frequency at the wave number vector $\wnk$ and grid size $h$. 
The superscripts $(1)$, $(2)$ correspond to permittivity matrix $\mattrivial$ 
and $\matcrossdof$, respectively. As is shown in Table \ref{tab:frequency_difference}, 
the difference is rather minor.

\begin{table}[htbp]
\setlength{\tabcolsep}{3pt}
\centering
\caption{Maximum and mean relative difference of frequencies computed with two types 
of $\Mcal_\varepsilon$ in Subsection \ref{sse:num:band} over the complete band.}
\label{tab:frequency_difference}
\begin{tabular}{|c|c|c|c|c|}
\hline
Lattice type & SC-CURV & FCC & BCC-SG & BCC-DG \\
\hline
$\Delta \omega_{\max}$ & $2.93 \times 10^{-3}$ & $2.76 \times 10^{-3}$ 
& $4.71 \times 10^{-3}$ & $4.72 \times 10^{-3}$ \\
\hline
$\Delta \omega_{\text{mean}}$ & $9.88 \times 10^{-4}$ & $1.08 \times 10^{-3}$ 
& $1.47 \times 10^{-3}$ & $1.08 \times 10^{-3}$ \\
\hline
\end{tabular}
\end{table}

\subsection{Precision Tests}\label{sse:num:highord}
Tables \ref{tab:num:convergence_type1} and \ref{tab:num:convergence_type2} present the first 8 smallest non-zero frequencies of the BCC-SG lattice in anisotropic media, computed with mesh resolutions $N=16, 32, 64,128$ and a 10th-order finite difference stencil. We choose both $\Mcal=\mattrivial$ and $\Mcal=\matcrossdof$ to evaluate the relative errors and the corresponding average convergence rates. For certain cases, both discretization methods achieve approximately second-order convergence, with $\Mcal=\matcrossdof$ yielding slightly higher accuracy than $\Mcal=\mattrivial$. However, for specific eigenvalues (the 3rd and 4th eigenvalues in the tables), the convergence rates may fall short of expectations, which is attributed to the limited mesh resolution or the poor regularity of the eigenfunctions.

\begin{table}[htbp]
\setlength{\tabcolsep}{3pt}
\centering
\caption{BCC-SG lattice in anisotropic media with wave number vector $\wnk=(\pi,\pi,\pi)$, convergence analysis of the first 8 eigenvalues, $\Mcal=\mattrivial$.}
\label{tab:num:convergence_type1}
\begin{tabular}{|c|c|c|c|c|}
\hline
\textbf{Eigenvalue} & \multicolumn{3}{c|}{\textbf{Absolute Difference 
$|\omega^2_h - \omega^2_{2h}|$}} & \textbf{Average} \\
\cline{2-4}
\textbf{Index} & \textbf{$h=1/32$} & \textbf{$h=1/64$} & \textbf{$h=1/128$} & \textbf{Order} \\
\hline
1 & 2.58e-02 & 7.70e-03 & 2.64e-03 & 1.64 \\ \hline
2 & 3.09e-02 & 1.10e-02 & 3.43e-03 & 1.59 \\ \hline
3 & 1.65e-02 & 1.89e-02 & 8.63e-03 & 0.47 \\ \hline
4 & 9.54e-03 & 1.62e-02 & 1.02e-02 & -0.05 \\ \hline
5 & 4.71e-02 & 1.45e-02 & 4.85e-03 & 1.64 \\ \hline
6 & 3.15e-02 & 2.10e-02 & 7.67e-03 & 1.02 \\ \hline
7 & 3.85e-02 & 2.06e-02 & 6.66e-03 & 1.27 \\ \hline
8 & 5.69e-02 & 2.10e-02 & 5.73e-03 & 1.66 \\ \hline
\end{tabular}
\end{table}

\begin{table}[htbp]
\setlength{\tabcolsep}{3pt}
\centering
\caption{BCC-SG lattice in anisotropic media with wave number vector $\wnk=(\pi,\pi,\pi)$, convergence analysis of the first 8 eigenvalues, $\Mcal=\matcrossdof$.}
\label{tab:num:convergence_type2}
\begin{tabular}{|c|c|c|c|c|}
\hline
\textbf{Eigenvalue} & \multicolumn{3}{c|}{\textbf{Absolute Difference 
$|\omega^2_h - \omega^2_{2h}|$}} & \textbf{Average} \\
\cline{2-4}
\textbf{Index} & \textbf{$h=1/32$} & \textbf{$h=1/64$} & \textbf{$h=1/128$} & \textbf{Order} \\
\hline
1 & 2.90e-02 & 8.21e-03 & 2.90e-03 & 1.66 \\ \hline
2 & 3.31e-02 & 1.21e-02 & 3.81e-03 & 1.56 \\ \hline
3 & 1.91e-02 & 2.30e-02 & 1.08e-02 & 0.41 \\ \hline
4 & 1.39e-02 & 1.58e-02 & 1.27e-02 & 0.06 \\ \hline
5 & 5.01e-02 & 2.16e-02 & 6.28e-03 & 1.50 \\ \hline
6 & 4.01e-02 & 2.59e-02 & 9.60e-03 & 1.03 \\ \hline
7 & 5.10e-02 & 2.43e-02 & 8.81e-03 & 1.27 \\ \hline
8 & 7.02e-02 & 2.55e-02 & 6.14e-03 & 1.76 \\ \hline
\end{tabular}
\end{table}

To verify the accuracy of the high-order stencils, we construct a test case where the dielectric coefficient $\varepsilon$ is smooth with respect to $\vec{x}$:
\begin{align}\label{display:num:testsmooth}
    \varepsilon(x,y,z)=8.9\sin(2\pi(x+y+z))+13.
\end{align}
Table \ref{tab:num:convergence_highord} presents the convergence results for the first eight eigenvalues, computed using a 10th-order finite difference stencil with $m=5$ at the wave vector $\wnk=(\pi,\pi,\pi)$. As shown in the table, the high-order finite difference stencils successfully achieve the theoretical convergence rate when the coefficients are smooth. Similar results hold for stencils of orders 2, 4, 6, and 8: these results are omitted here for simplicity.
\begin{table}[htbp]
\setlength{\tabcolsep}{6pt}
\centering
\caption{SC lattice where dielectric coefficient $\varepsilon=\varepsilon(\vec x)$ is smooth \eqref{display:num:testsmooth}, wave number vector $\wnk=(\pi,\pi,\pi)$, convergence analysis of the first 8 eigenvalues.}
\label{tab:num:convergence_highord}
\begin{tabular}{|c|c|c|c|c|}
\hline
\textbf{Eigenvalue} & \multicolumn{3}{c|}{\textbf{Absolute Difference 
$|\omega^2_h - \omega^2_{2h}|$}} & \textbf{Average} \\
\cline{2-4}
\textbf{Index} & \textbf{$h=1/32$} & \textbf{$h=1/64$} & \textbf{$h=1/128$} & \textbf{Order} \\
\hline
1 & 4.87e-09 & 5.36e-12 & 6.22e-15 & 9.79 \\ \hline
2 & 4.87e-09 & 5.36e-12 & 7.41e-15 & 9.66 \\ \hline
3 & 1.97e-08 & 2.23e-11 & 2.53e-14 & 9.79 \\ \hline
4 & 1.97e-08 & 2.23e-11 & 3.14e-14 & 9.63 \\ \hline
5 & 1.97e-08 & 2.23e-11 & 3.31e-14 & 9.59 \\ \hline
6 & 3.60e-08 & 4.13e-11 & 3.51e-14 & 9.98 \\ \hline
7 & 3.60e-08 & 4.13e-11 & 4.14e-14 & 9.87 \\ \hline
8 & 3.60e-08 & 4.14e-11 & 3.31e-14 & 10.03 \\ \hline
\end{tabular}
\end{table}

\subsection{Computational Performance}\label{sse:num:perf}
In Table \ref{tab:iterations_full} we present iteration counts required to 
achieve convergence for the types of PCs discussed in Subsection \ref{sse:num:band}. 
To be precise, we compute both the average and the standard deviation of iterations at different $\wnk$'s, showing the numbers 
don't vary too much. More iterations are required when computing anisotropic system than that of computing 
isotropic system with same lattice type and tolerance. 
Specifically, $\Mcal_\varepsilon=\matcrossdof$ requires fewer iterations than $\Mcal_\varepsilon=\mattrivial$. 
\begin{table}[htbp]
\centering
\caption{Average iterations (and standard deviation) for computing complete band structure. 
Grid division $N=120$, first 10 eigenpairs are computed with tolerance $10^{-5}$.}
\begin{tabular}{|c|c|c|c|c|}
\hline
\ & SC-CURV & FCC & BCC-SG & BCC-DG \\
\hline
Isotropic & 32.8 (4.0) & 42.2 (6.9) & 55.0 (6.9) & 58.0 (6.6) \\
\hline
Anisotropic $\Mcal_\varepsilon=\mattrivial$  & 52.3 (5.1) & 54.0 (7.4) & 68.7 (8.7) & 76.5 (8.5)\\
\hline
Anisotropic $\Mcal_\varepsilon=\matcrossdof$ & 45.3 (4.2) & 51.3 (7.9) & 65.8 (8.8) & 74.7 (8.1)\\
\hline
\end{tabular}
\label{tab:iterations_full}
\end{table}

Next, we provide a quantitative measure of performance enhancement brought by GPU acceleration. The linear algebra operations are done by NumPy and CuPy packages on CPU and GPU, respectively. From Table \ref{tab:speedup}, 
we observe that GPU acceleration achieves an average speedup of over 40 times. 

\begin{table}[htb]
\caption{CPU and GPU runtime (seconds) for the first 10 eigenvalues of SC-CURV, FCC, BCC-SG, BCC-DG at wave number vector $\wnk=(\pi,\pi,\pi)$, 
grid division $N=100,120$.}
\setlength{\tabcolsep}{2pt}
\centering
\begin{tabular}{|c|c|c|c|c|c|c|c|c|}
\hline
\multirow{2}{*}{DoFs}  & \multicolumn{4}{c|}{\textbf{Isotropic SC-CURV}} & \multicolumn{4}{c|}{\textbf{Anisotropic SC-CURV}} \\
\cline{2-9}
& Steps & GPU time & CPU time & Speed up & Steps & GPU time & CPU time & Speed up \\
\hline
$3\times 100^3$ & 31 & 10.79 & 432.71 & 40.10 & 50 & 16.67 & 666.39 & 39.98 \\
\hline
$3\times 120^3$ & 34 & 19.85 & 835.91 & 42.11 & 49 & 28.67 & 1297.50 & 45.26 \\
\hline
\multirow{2}{*}{DoFs}  & \multicolumn{4}{c|}{\textbf{Isotropic FCC}} & \multicolumn{4}{c|}{\textbf{Anisotropic FCC}} \\
\cline{2-9}
& Steps & GPU time & CPU time & Speed up & Steps & GPU time & CPU time & Speed up \\
\hline 
$3\times 100^3$ &  45 & 16.00 & 623.77 & 39.01 & 56 & 19.59 & 852.17 & 43.50 \\
\hline
$3\times 120^3$ & 45 & 27.71 & 1197.90 & 43.23 & 55 & 34.15 & 1506.35 & 44.11 \\
\hline  
\multirow{2}{*}{DoFs}  & \multicolumn{4}{c|}{\textbf{Isotropic BCC-SG}} & \multicolumn{4}{c|}{\textbf{Anisotropic BCC-SG}} \\
\cline{2-9}
& Steps & GPU time & CPU time & Speed up & Steps & GPU time & CPU time & Speed up \\
\hline
$3\times 100^3$ &  57 & 19.85 & 835.33 & 42.08 & 68 & 24.14 & 1020.15 & 42.26 \\
\hline
$3\times 120^3$ & 47 & 27.96 & 1295.39 & 46.33 & 65 & 41.08 & 1816.56 & 44.22 \\
\hline
\multirow{2}{*}{DoFs}  & \multicolumn{4}{c|}{\textbf{Isotropic BCC-DG}} & \multicolumn{4}{c|}{\textbf{Anisotropic BCC-DG}} \\
\cline{2-9}
& Steps & GPU time & CPU time & Speed up & Steps & GPU time & CPU time & Speed up \\
\hline 
$3\times 100^3$ & 89 & 26.83 & 1056.57 & 39.98 & 77 & 25.25 & 1038.00 & 41.11 \\
\hline
$3\times 120^3$ & 86 & 44.61 & 1903.39 & 43.20 & 76 & 43.55 & 1971.94 & 45.28 \\
\hline       
\end{tabular}

\label{tab:speedup}
\end{table}

\subsection{Robustness under Extreme Cases}\label{sse:num:custom}
We construct an ill-conditioned $\varepsilon_1$ as $\varepsilon_1=U\diag(10^{-1},10^{-3},10^{-5})U^\conjT$, 
where $U$ is a randomly generated unitary matrix. Under the setting of SC-CURV lattice, 
the system is computed with both $\Mcal_\varepsilon=\mattrivial$ 
and $\Mcal_{\varepsilon}=\matcrossdof$. %The approximate asymptotic damping factor $\theta\in(0,1)$ is introduced to quantify the convergence rate, which is evaluated by applying linear regression to the logarithms of residuals. 
The dampening history of the residuals is documented in Figure \ref{fig:extreme1}, where we observe a smooth convergence. Moreover, $\matcrossdof$ has better numerical performance in avoiding stagnation under extreme cases.

This customized matrix $\varepsilon_1$, in general, does not satisfy Assumption \ref{display:proof_assump_sdd} or \ref{display:proof_assump_nonzero}, which further shows the robustness of our scheme.

\begin{figure}
\centering
\includegraphics[width=0.85\linewidth]{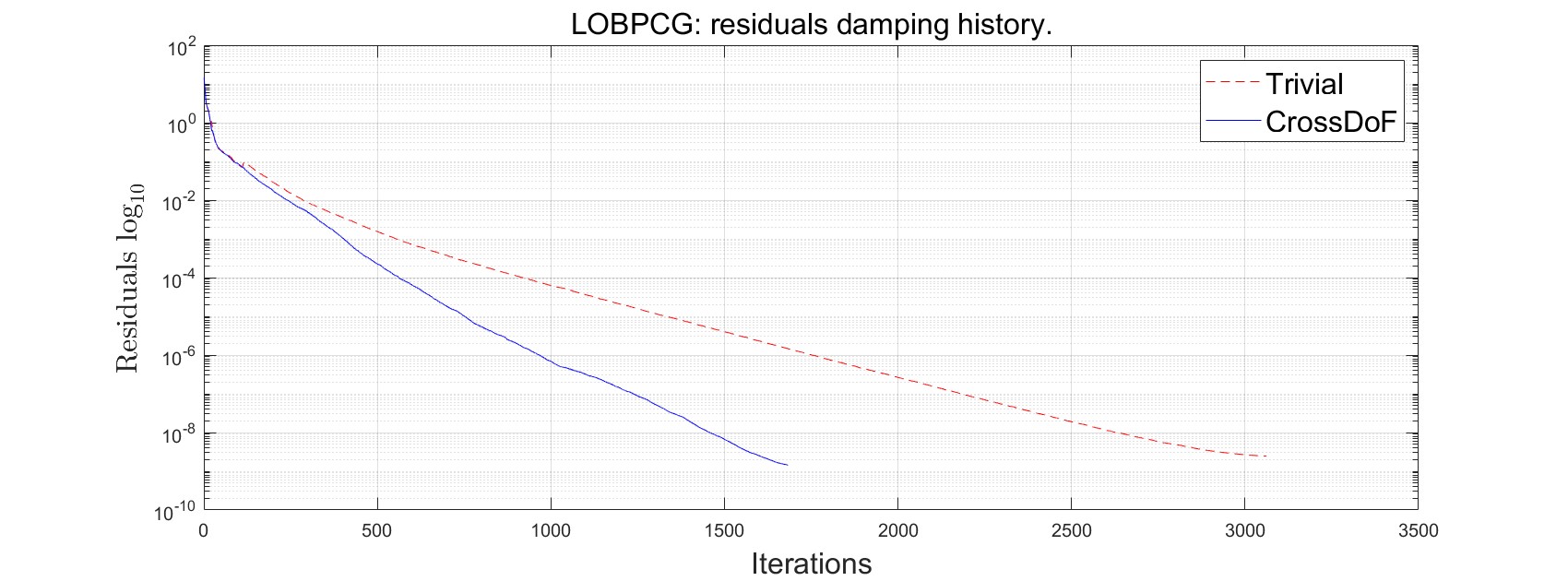}
\caption{Convergence history where $\varepsilon_1$ is very ill-conditioned. 
Anisotropic SC-CURV lattice at $\wnk=(\frac{1}{7},\frac{3}{5},\frac{4}{13})\pi$, 
grid division $N=120$.}
\label{fig:extreme1}
\end{figure}

\section{Conclusions}\label{se:con}
In this paper, we have presented a robust, GPU-accelerated framework for solving 3D
Maxwell eigenproblems in photonic crystals with anisotropic media. Our key innovations
include a novel, provably positive definite discretization for the anisotropic permittivity 
tensor and a rigorous theoretical completion of the kernel compensation method, which now includes a precise characterization of the null space $\Hcal$, an explicit rule for selecting the penalty 
coefficient $\gamma$, an approach to matrix-free operations via 3D DFT and extensions to high order discretizations. Numerical experiments on various lattice structures demonstrate practicality and efficiency, 
where an average speedup of more than $40\times$ is achieved on GPU. An ill-conditioned 
example also shows the robustness. Our future work will focus on extending this powerful 
framework to more complex, nonlinear
eigenvalue problems in photonic crystals.

%----------------------------------------------------------------------------------------
%	Bibliography
%----------------------------------------------------------------------------------------
\bibliographystyle{abbrv}
\bibliography{cas-refs}{}

\end{document}